\documentclass[12pt,reqno]{amsart}
\usepackage[english]{babel}
\usepackage{mathrsfs, csquotes}
\usepackage{amsfonts,amssymb,amsmath,mathtools,commath,braket,amssymb}
\usepackage{color} \definecolor{bleu_sombre}{rgb}{0,0,0.6}  \definecolor{rouge_sombre}{rgb}{0.8,0,0}\definecolor{vert_sombre}{rgb}{0,0.6,0}
\usepackage[plainpages=false,colorlinks,linkcolor=bleu_sombre,
citecolor=rouge_sombre,urlcolor=vert_sombre,breaklinks]{hyperref}
\usepackage{enumerate}

\theoremstyle{plain}
\newtheorem{theorem}{Theorem}[section]
\newtheorem{lemma}[theorem]{Lemma}
\newtheorem{corollary}[theorem]{Corollary}
\newtheorem{proposition}[theorem]{Proposition}

\theoremstyle{definition}
\newtheorem{remark}[theorem]{Remark}
\newtheorem{notation}{Notation}

\newtheorem{definition}[theorem]{Definition}
\newtheorem{assumption}[theorem]{Assumption}
\newtheorem{observation}[theorem]{Observation}

\renewcommand{\leq}{\leqslant}	\renewcommand{\geq}{\geqslant}

\def\CC{{\mathbb C}}
\def\RR{{\mathbb R}}
\def\NN{{\mathbb N}}
\def\ZZ{{\mathbb Z}}

\def\({\left(}
\def\){\right)}
\def\<{\left\langle}
\def\>{\right\rangle}
\def\le{\leqslant}
\def\ge{\geqslant}
\def\n{\mathbf{n}}

\def\eps{\varepsilon}

\DeclareMathOperator{\p}{\mathbf{p}}

\numberwithin{equation}{section}

\newcommand{\dd}{\mathrm{d}}
\newcommand{\num}{F_h}
\newcommand{\den}{G_h}
\newcommand{\be}{\begin{equation}}
\newcommand{\ee}{\end{equation}}
\newcommand{\bea}{\begin{eqnarray}}
\newcommand{\eea}{\end{eqnarray}}
\newcommand{\bee}{\begin{eqnarray*}}
\newcommand{\eee}{\end{eqnarray*}}

\def\eps{\varepsilon}

\def\pa{\partial}

\def\n{\mathbf{n}}

\def\DP{\mathscr{P}_h}
\def\Nb{{N_\mathcal{B}}}
\def\Nh{{N_{\mathcal{H}}}}
\def\disth{\mathsf{dist}_{\mathcal{H}}}
\def\disthd{\mathsf{dist}_{\mathcal{H},\delta}}
\def\distb{\mathsf{dist}_{\mathcal{B}}}

\begin{document}

\title[Spectrum of the Dirichlet-Pauli operator]{On the semiclassical spectrum \\ of the Dirichlet-Pauli operator}
\author[J.-M. Barbaroux]{J.-M. Barbaroux}
\email{barbarou@univ-tln.fr}
\address[J.-M. Barbaroux]{Aix Marseille Univ, Universit\'e de Toulon, CNRS, CPT, Marseille, France}
\author[L. Le Treust]{L. Le Treust}
\email{loic.le-treust@univ-amu.fr}
\address[L. Le Treust]{Aix Marseille Univ, CNRS, Centrale Marseille, I2M, Marseille, France}
\author[N. Raymond]{N. Raymond}
\email{nicolas.raymond@univ-angers.fr }
\address[N. Raymond]{LAREMA, Universit\'e d'Angers, Facult\'e des Sciences, D\'epartement de Math\'ematiques, 49045 Angers cedex 01, France}
\author[E. Stockmeyer]{E. Stockmeyer}
\email{stock@fis.puc.cl}
\address[E. Stockmeyer]{Instituto de F\'isica, Pontificia Universidad Cat\'olica de Chile, Vicu\~na Mackenna 4860, Santiago 7820436, Chile.}

\subjclass[2010]{35P15, 81Q10, 35Jxx}
\keywords{Pauli operator, magnetic Cauchy-Riemann operators, semiclassical analysis}

\begin{abstract}
This paper is devoted to semiclassical estimates of the eigenvalues of the Pauli operator on a bounded open set whose boundary carries Dirichlet conditions. Assuming that the magnetic field is positive and a few generic conditions, we establish the simplicity of the eigenvalues and provide accurate asymptotic estimates involving Segal-Bargmann and Hardy spaces associated with the magnetic field.
\end{abstract}
\maketitle
\tableofcontents
\section{Introduction}
\label{sec:intro}
In this article we consider the magnetic Pauli operator defined on a bounded and simply-connected domain $\Omega\subset \RR^2$
subject to Dirichlet boundary conditions. This operator  is the model
Hamiltonian of a non-relativistic spin-$\tfrac{1}{2}$ particle, constraint to move in $\Omega$,
interacting with a magnetic field that is perpendicular to the plane.

Formally the Pauli operator acts on two-dimensional spinors and it is
given by $$ \DP =[\sigma\cdot (-ih\nabla-A)]^2,$$ where $h>0$ is a
semiclassical parameter and $\sigma$ is a two-dimensional vector whose
components are the Pauli matrices $\sigma_1$ and $ \sigma_2$. The
magnetic field $B$ enters in the operator through an associated magnetic
vector potential $A=(A_1,A_2)$ that satisfies
$\partial_1 A_2 -\partial_2 A_1=B$. Assuming that the magnetic field
is positive and few other mild conditions we provide precise
asymptotic estimates for the low energy eigenvalues of $\DP$ in the
semiclassical limit (i.e., as $h\to 0$). 

Let us roughly explain our results. Let $\lambda_k(h)$ be the $k$-th
eigenvalue of $\DP$ counting multiplicity. Assuming that the boundary
of $\Omega$ is $\mathscr{C}^2$, we show that there exist 
$\alpha>0$, $\theta_0\in (0,1]$ such that the following holds: For all $k\in \NN^*$, there exists $C_k>0$ such that, as $h\to 0$,
 \begin{align*}
\theta_0 C_k h^{-k+1} e^{-2\alpha/h }(1+o(1)) \le \lambda_k(h)\le C_k
   h^{-k+1} e^{-2\alpha/h } 
(1+o(1))\,.
 \end{align*} 
In particular, this result establishes the simplicity of the eigenvalues in this regime.
The constants $\alpha>0$ and $C_k$ are directly related to the magnetic field and the geometry of $\Omega$ and $C_k$ is expressed in terms of Segal–Bargmann and Hardy norms that are
naturally associated to the magnetic field. In the case when $\Omega$
is a disk and $B$ is radially symmetric we compute $C_k$ explicitly
and find that $\theta_0=1$. This improves by large the known results about the Dirichlet-Pauli operator \cite{E96,HP17} (for details see Section \ref{sec.literature}).

These results may be reformulated in terms of the large magnetic field limit by a simple scaling argument. Indeed,
$\mu_k(b)=b^2\lambda_k(1/b)$, where $\mu_k(b)$ is the $k$-th
eigenvalue of $[\sigma\cdot (-i\nabla-bA)]^2$.

Our results can also be used to describe the spectrum of the magnetic Laplacian with constant magnetic field $B_0$. For instance, when $\Omega$ is bounded, strictly convex with a boundary of class $\mathscr{C}^{1,\gamma}$ ($\gamma>0$), the $k$-th eigenvalue of $(-ih\nabla-A)^2$ with Dirichlet boundary conditions, denoted by $\mu_{k}(h)$, satisfies, for some $c,C>0$ and $h$ small enough,
\begin{equation}\label{eq.B=ct2}
B_{0}h+ch^{-k+1}e^{-2\alpha/h}\leq \mu_{k}(h)\leq B_{0}h+Ch^{-k+1}e^{-2\alpha/h}\,.
\end{equation}
In particular, the first eigenvalues of the magnetic Laplacian are simple in the semiclassical limit. This asymptotic simplicity was not known before and \eqref{eq.B=ct2} is the most accurate known estimate of the magnetic eigenvalues in the case of the constant magnetic field and Dirichlet boundary conditions (See \cite[Section 4]{HelfferMorame} and Section \ref{sec.literature}).

Our study  presents a new approach that establishes several
connections with various aspects of analysis as Cauchy-Riemann
operators, uniformisation, and, to some extent, Toeplitz operators. We
may hope that this work will cast a new light on the magnetic Schr\"odinger operators.

\newpage

 \subsection{Setting and main results.}

	Let $\Omega\subset\RR^2$ be an open set. All along the paper $\Omega$ will satisfy the following assumption.
	
	\begin{assumption}\label{asum:setom}
	$\Omega$ is bounded and simply connected.
	\end{assumption}

	  Consider a magnetic field $B\in \mathcal{C}^{\infty}(\overline{\Omega}, \RR)$. An associated vector potential $A : \overline{\Omega}\longrightarrow 	\RR^2$ is a function such that
	 \[
	 	B = \pa_1 A_2 - \pa_2A_1\,.
	\]
	We will  use the following special choice of vector potential. 
	
	\begin{definition}\label{defi.gauge}
	Let $\phi$ be the unique (smooth) solution of 
			\be \label{eq:0modeFlux}
				\begin{split}
					\Delta \phi &= B, \mbox{ on }\Omega\,,\\
					\phi &= 0,\mbox{ in }\pa \Omega\,.
				\end{split}
			\ee
	The vector field $A = (-\pa_2 \phi, \pa_1\phi)^T:=\nabla\phi^\perp$ is a vector potential associated with $B$.
	\end{definition}
	In this paper, $B$ will be positive (and thus $\phi$ subharmonic) so that \[\max_{x\in\overline{\Omega}}\phi=\max_{x\in\partial\Omega}\phi=0\,.\]
	In particular, the minimum of $\phi$ will be negative and attained in $\Omega$. Note also that the exterior normal derivative of $\phi$, denoted by $\partial_{\n}\phi$, is positive on $\partial\Omega$  if $\Omega$ is $\mathscr{C}^2$ \cite[Hopf's Lemma, Section 6.4.2]{evans1998partial}. 

	\begin{notation} 
	We denote $\langle \cdot,\cdot \rangle$ the $\CC^n$ ($n\geq1$) scalar product (antilinear w.r.t. the left argument), $\langle \cdot,\cdot\rangle_{L^2(U)}$ the $L^2$ scalar product on the set $U$, $\|\cdot\|_{L^2(U)}$ the $L^2$ norm on $U$ and $\|\cdot\|_{L^\infty(U)}$ the $L^\infty$ norm on $U$. We use $o$ and $\mathscr{O}$ for the standard Laudau symbols. 
	\end{notation}

\subsection{The Dirichlet-Pauli operator}	
This paper is devoted to the Dirichlet-Pauli operator $(\DP,\mathsf{Dom}(\DP))$ defined for all $h>0$ on 
	\[\mathsf{Dom}(\DP):= H^2(\Omega; \CC^2)\cap H^1_0(\Omega; \CC^2)\,,\] 
	and whose action is given by the second order differential operator
	\be\label{eq:defDP}
		\DP = [\sigma\cdot (\p -A)]^2 = 
		\begin{pmatrix}
			|\p-A|^2-hB&0\\
			0&|\p-A|^2+hB
		\end{pmatrix} = \begin{pmatrix} \mathscr{L}^-_{h} &0\\0& \mathscr{L}^+_{h}\end{pmatrix}\,.
	\ee
	Here $\p=-ih\nabla$, and  
	\[
	|\p-A|^2 := (\p-A)\cdot(\p-A) = -h^2\Delta-A\cdot \p-\p\cdot A+|A|^2\,,
	\]  and $\sigma=(\sigma_1, \sigma_2, \sigma_3)$ are the Pauli matrices:
		\[
	\sigma_1 = \left(
		\begin{array}{cc}
			0&1\\1&0
		\end{array}
	\right),
	\quad
	\sigma_2 = \left(
		\begin{array}{cc}
			0&-i\\i&0
		\end{array}
	\right),
	\quad
	\sigma_3 = \left(
		\begin{array}{cc}
			1&0\\0&-1
		\end{array}
	\right)\, ,
	\]
and $\sigma\cdot \mathbf{x} = \sigma_1\mathbf{x}_1 + \sigma_2\mathbf{x}_2+\sigma_3\mathbf{x}_3$ for $\mathbf{x} = (\mathbf{x}_1,\mathbf{x}_2,\mathbf{x}_3)$ 
and 
$\sigma\cdot \mathbf{x} = \sigma_1\mathbf{x}_1 + \sigma_2\mathbf{x}_2$ for $\mathbf{x} = (\mathbf{x}_1,\mathbf{x}_2)$.
	In terms of quadratic form, we have by partial integration, for all $\psi\in \mathsf{Dom}(\DP)$,
			\begin{equation}\label{def:quadform}
				\langle \psi,\DP \psi\rangle_{L^2(\Omega)}  =\|\sigma\cdot(\p-A) \psi\|^2_{L^2(\Omega)}=  \|(\p-A) \psi\|^2_{L^2(\Omega)} - \<\psi,\sigma_3 h B\psi\>_{L^2(\Omega)}\, .
			\end{equation}
	 Note that we have the following relation, for all $\mathbf{x}, \mathbf{y}\in\RR^3$,
	\begin{equation}\label{eq.prod}
	(\sigma\cdot\mathbf{x})(\sigma\cdot\mathbf{y})=\mathbf{x}\cdot\mathbf{y}1_{2}+i\sigma\cdot(\mathbf{x}\times\mathbf{y})\,,	
         \end{equation}
         where $1_2$ is the identity matrix of $\CC^2$. 
         The operator $\DP$ is selfadjoint and has compact resolvent. This paper is mainly devoted to the investigation of the lower eigenvalues of $\DP$.  
         \begin{notation}
		Let $(\lambda_k(h))_{k\in \NN^*}$ ($h>0$) denote the increasing sequence of eigenvalues of the operator $\DP$, each one being repeated according to its multiplicity. By the min-max theorem,
		\begin{equation}\label{eq:eigDP}
			\lambda_k(h) = \inf_{
				\begin{array}{c}
					V\subset \mathsf{Dom}(\DP)\,,\\
					\dim V = k\,,
				\end{array}
				}
				\sup_{
				\begin{array}{c}
					\psi\in V\setminus\{0\}\,,
				\end{array}
				}
				\frac{\norm{\sigma\cdot(\p-A)\psi}_{L^2(\Omega)}^2}{\norm{\psi}_{L^2(\Omega)}^2}\,.
		\end{equation}
		 Under the assumption that $B>0$ on $\overline{\Omega}$, the lowest eigenvalues of $\DP$ are the eigenvalues of $\mathscr{L}^-_{h}$.
		 More precisely, our main result states that for any fixed $k\in \NN^*$ and $h>0$ small enough, $\lambda_k(h) $ is the $k$-th eigenvalue of the Schr\"odinger operator  $\mathscr{L}^-_{h}$.
         \end{notation}
	
\subsection{Results and relations with the existing literature}	

\subsubsection{Main theorem}\label{sec.mt}

\begin{notation}\label{not.BH}
	Let us denote by $\mathscr{H}(\Omega)$ and $\mathscr{H}(\CC)$ the sets of holomorphic functions on $\Omega$ and $\CC$.
	We consider the following (anisotropic)  Segal-Bargmann space
	\[\mathscr{B}^2(\CC) = \{u\in\mathscr{H}(\mathbb{C}) :\Nb(u)<+\infty\}\,,\] 
	where
	\[
	\Nb(u)=\left(\int_{\mathbb{R}^2} \left|u \left(y_1+iy_2\right)\right|^2e^{-\mathsf{Hess}_{x_{\min}}\phi(y, y)} \dd y\right)^{1/2}\,.
	\]
	We also introduce a weighted Hardy space
 \[\mathscr{H}^2(\Omega)=\{u\in\mathscr{H}(\Omega) : \Nh (u)<+\infty\}\,,\] 
	where
	\[
	\Nh (u)=\left(\int_{\partial \Omega} \left|u \left(y_1+iy_2\right)\right|^2\partial_{\textbf{n}}\phi \dd y\right)^{1/2}\,.
	\]
	Here, $x_{\rm min}\in \Omega$  and $\mathsf{Hess}_{x_{\min}}\phi\in \mathbb{R}^{2\times 2}$ are defined in Theorem \ref{theo.main} below, $\n(s)$ is the outward pointing unit normal to $\Omega$, and $\pa_\n \phi(s)$ is the normal derivative of $\phi$ on $\partial \Omega$ at $s\in\partial\Omega$.
	We also define for $P\in \mathscr{H}^2(\Omega)$, $A\subset \mathscr{H}^2(\Omega)$,
	\[\begin{split}
	\disth(P,A) = \inf\left\{
		\Nh(P-Q)\,, \mbox{ for all }Q\in A
	\right\}\,,
\end{split}\]
	and
	for $P\in \mathscr{B}^2(\CC)$, $A\subset \mathscr{B}^2(\CC)$,
	\[\begin{split}
	\distb(P,A) = \inf\left\{
		\Nb(P-Q)\,, \mbox{ for all }Q\in A
	\right\}\,.
\end{split}\]
	\end{notation}

The main results of this paper are gathered in the following theorem.
\begin{theorem}\label{theo.main}
We define
	\[\phi_{\min} = \min_{x\in\overline{\Omega}} \phi\,.\] 
	Assume that $\Omega$ is $\mathscr{C}^2$, satisfies {\rm Assumption \ref{asum:setom}}, and
	\begin{enumerate}[\rm (a)]
		\item\label{eq.a1} $B_0 := \inf\{B(x),\,x\in\Omega\}>0$,
		\item\label{eq.a2}  the minimum of $\phi$ is attained at a unique point $x_{\min}$,
		\item\label{eq.a3}  the minimum is non-degenerate, i.e., the Hessian matrix $\mathsf{Hess}_{x_{\min}}\phi$ at $x_{\min}$ (or $z_{\min}$ if seen as a complex number) is positive definite.
	\end{enumerate}
Then, there exists $\theta_0\in(0,1]$ such that for all fixed $k\in\mathbb{N}^*$ ,
\begin{enumerate}[\rm (i)]
\item \label{theo.upb}
$\lambda_{k}(h)\leq  C_{\rm sup}(k) h^{-k+1}e^{2\phi_{\min}/h}(1+o_{h\to0}(1)),$ 
with
\[C_{\rm sup}(k)=2\left(\frac{\disth\left((z-z_{\min})^{k-1},\mathscr{H}^2_k(\Omega)\right)}{\distb\left(z^{k-1},\mathcal{P}_{k-2}\right)}\right)^2\,,\]
where $\mathcal{P}_{k-2} = {\rm span}\left(1,\dots , z^{k-2}\right)\subset \mathscr{B}^2(\CC)$, $\mathcal{P}_{-1}=\{0\}$ and
\begin{equation}\label{space.Hk}
\mathscr{H}^2_k(\Omega) = \{u\in \mathscr{H}^2(\Omega),\ u^{(n)}(z_{\min}) = 0, \mbox{ for }n\in\{0,\dots,k-1\}\}\,.
\end{equation}
\item \label{theo.lob} $\lambda_{k}(h)\geq  C_{\inf}(k) h^{-k+1}e^{2\phi_{\min}/h}(1+o_{h\to0}(1)),$
with
\[C_{\rm inf}(k)=C_{\rm sup}(k)\theta_0\,.\]
\end{enumerate}
A precise definition of $\theta_0$ is given in Remark \ref{rem.theta0}.

Assuming that $\Omega$ is the  disk of radius $1$ centered at $0$, and that $B$ is radial, we have
\[\begin{split}
&C_{\sup}(k)=C_{\inf}(k)=C^\mathrm{rad}(k)=
\frac{B(0)^k\Phi}{2^{k-2}(k-1)!}
\,,\ 
(\theta_0=1)\,,
\\&\Phi=\frac{1}{2\pi}\int_{\Omega}B(x)\dd x=\frac{1}{2\pi}\int_{\pa\Omega}\partial_{\mathbf{n}}\phi\dd s\,.
\end{split}\]
\end{theorem}

\begin{remark}\label{rem.nondeg}
 Assume that $B=B_{0}>0$ and that $\Omega$ is strictly convex, then $\phi$ has a unique and non-degenerate minimum (see \cite{K85, K85b} and also \cite[Proposition 7.1 and below]{HP17}). Thus, our assumptions are satisfied in this case.
\end{remark}

\begin{remark}\label{rem.choihardy}
	The main properties of the space $\mathscr{H}^2(\Omega)$ can be found in \cite[Chapter $10$]{duren2000theory}. 
	Note that whenever $\partial\Omega$ is supposed to be Dini continuous (in particular $\mathscr{C}^{1,\alpha}$ boundaries, with $\alpha>0$, are allowed), the set $W^{1,\infty}(\Omega)\cap \mathscr{H}^2(\Omega)$ is dense in $\mathscr{H}^2(\Omega)$ (see Lemma \ref{lem.density}).
	This assumption is in particular needed in the proof of Theorem \ref{theo.main}\eqref{theo.upb} (see Remark \ref{rem.hardy_approx})\footnote{Note also that we do not use here the stronger notion of  Smirnov domain in which the set of polynomials in the complex variable is dense in $\mathscr{H}^2(\Omega)$ (see \cite[Theorem $10.6$]{duren2000theory}). Starlike domains and domains with analytic boundary are Smirnov domains.}. 
	The definition of Dini-continuous functions is recalled in the context of the boundary behavior of conformal maps in \cite[Section 3.3]{P92}. It is essentially an integrability property of the derivative of a parametrization of $\partial\Omega$. 
\end{remark}
\begin{remark}\label{rem.hilbertprojhardy}
The Cauchy formula \cite[Theorem $10.4$]{duren2000theory} and the Cauchy-Schwarz inequality ensure that
\[
	|u^{(n)}(z_{\rm min})|\leq \frac{n!}{2\pi \sqrt{\min_{\partial \Omega}\pa_\n\phi}}\Nh(u)\left(\int_{\partial \Omega}\frac{|\dd z|}{|z-z_{\min}|^{2(n+1)}}\right)^{1/2}\,,
\]
for $n\in \NN$ and $u\in \mathscr{H}^2(\Omega)$ (see  also the proof of Lemma \ref{lem:L2esti1}). This ensures that the space $\mathscr{H}^2_k(\Omega)$ defined in \eqref{space.Hk} is a closed vector subspace of $\mathscr{H}^2(\Omega)$ and that $\disth \left((z-z_{\rm min})^{k-1},\mathscr{H}^2_k(\Omega)\right)>0$ (see \cite[Corollary $5.4$]{MR2759829}) since $(z-z_{\rm min})^{k-1}\notin \mathscr{H}^2_k(\Omega)$.
\end{remark}
\begin{remark}\label{rem.sym}
In the case when $B$ is radial on the unit disk  $\Omega=D(0,1)$, we get, using Fourier series, that $(z^n)_{n\geq0}$ is an orthogonal basis for $\Nb$ and $\Nh$ which are up to  normalization factors, the Szeg\"o polynomials \cite[Theorem $10.8$]{duren2000theory}. In particular, $\mathscr{H}^2_k(\Omega)$ is $\Nh$-orthogonal to $z^{k-1}$ so that
\[
	\disth \left(z^{k-1},\mathscr{H}^2_k(\Omega)\right)^2 = \Nh(z^{k-1})^2 = \int_{\partial \Omega} \pa_\n\phi = 2\pi\Phi\,.
\]
In addition, $\mathcal{P}_{k-2}$ is $\Nb$-orthogonal to $z^{k-1}$ so that
\[
	\distb \left(z^{k-1},\mathcal{P}_{k-2}\right)^2 = \Nb(z^{k-1})^2 = 
	2\pi\frac{2^{k-1}(k-1)!}{B(0)^k}
	\,,
\]
and the radial case part of Theorem \ref{theo.main} follows. 
\end{remark}
\begin{remark}
	The proof of the upper bound can easily be extended to the case where $\Omega$ is not necessarily simply connected (see Remark \ref{rem.hardy_approx}).
\end{remark}
\begin{remark}
	Theorem \ref{theo.main} is concerned with the asymptotics of each eigenvalues $\lambda_{k}(h)$ of the operator $\DP$, ($k\in\NN^*$)  as $h\to 0$. In particular,  $\lambda_k(h)$ tends to $0$ exponentially. Of course, this does not mean that all the eigenvalues go to $0$ uniformly with respect to $k$. For $h>0$, consider for example
	\[
		\nu_1(h) = \inf_{v\in H^1_0(\Omega; \CC)\setminus\{0\}}\frac{\braket{u, \mathscr{L}^+_{h} u}}{\norm{v}_{L^2(\Omega)}^2}\,,
	\]
	the lowest eigenvalue of the operator $\mathscr{L}^+_{h}$. For fixed $h>0$, there exists $k(h)\in\NN^*$ such that $\nu_1(h) = \lambda_{k(h)}(h)$. By \eqref{eq:defDP}, we have $\nu_1(h)\geq 2B_0h$ and thus $\nu_1(h)$ does not converge to $0$ with exponential speed. Actually, Theorem \ref{theo.main} ensures that
	\[
		\lim_{h\to0} {\rm card}\{j\in\NN^*\,, \lambda_{j}(h)\leq \nu_1(h)\} = +\infty\,,\qquad \lim_{h\to0}k(h)=+\infty\,.
	\]
	This accumulation of eigenvalues near $0$ in the semiclassical is related to the fact that the corresponding eigenfunctions are close to be functions in the Segal–Bargmann space $\mathscr{B}^2(\CC)$ which is of infinite dimension.
\end{remark}

\begin{remark}\label{rem.theta0}
	The constant $\theta_0$ introduced in Theorem \ref{theo.main} does not depend on $k\in\NN^*$ and is equal to $1$ in the radial case.
	We conjecture that the upper bounds in Theorem \ref{theo.main} \eqref{theo.upb} are optimal, that is  $\theta_0=1$ in the general case. 

	More precisely, let $\Omega$ be a $\mathscr{C}^2$ set satisfying {\rm Assumption \ref{asum:setom}}. We introduce
\[
	\mathcal{M}_\Omega := \{G:\overline\Omega\to \overline{ D(0,1)}\mbox{ biholomorphic s.t. } c_1\leq |G'(\cdot)|\leq c_2\,, \mbox{ for some }c_1,c_2>0\}\,.
\]
	Remark that $\mathcal{M}_\Omega$ is non-empty by the Riemann mapping theorem. Then, the constant $\theta_0$  can be defined by
\[\theta_0 := \frac{\min_{\pa D(0,1)}|(G^{-1})'(y)|\pa_{\n}\phi(G^{-1}(y))}{\max_{\pa D(0,1)}|(G^{-1})'(y)|\pa_{\n}\phi(G^{-1}(y))}\in(0,1]\,,\]
for some $G\in\mathcal{M}_\Omega$ (see Lemma \ref{lem:lowbound_4}).\footnote{
 We can even choose 
\[\tilde\theta_0 := \sup_{G\in \mathcal{M}_\Omega}\frac{\min_{\pa D(0,1)}|(G^{-1})'(y)|\pa_{\n}\phi(G^{-1}(y))}{\max_{\pa D(0,1)}|(G^{-1})'(y)|\pa_{\n}\phi(G^{-1}(y))}\,.\]
}

Actually, we can even see from our analysis that there is a class of magnetic fields for which $\theta_{0}=1$. 
We introduce

\begin{equation}\begin{split}
\mathcal{B} := \big\{& \check B\in \mathscr{C}^\infty(\overline{D(0,1)}; \RR_{+}^*)\,,\\
&	\exists \check\phi\in H^1_0(D(0,1);\RR)\,, \Delta  \check \phi = \check B\mbox{ on }D(0,1)\,, \pa^2_{\n s} \check \phi = 0\,\mbox{ on  }\partial D(0,1)\}\,.
\end{split}\end{equation}
Here, $\pa_s$ denotes the tangential derivative. 
%
Then, for any $\check B\in\mathcal{B}$ and $G\in\mathcal{M}_\Omega$, we get $\theta_0 = 1$ and
\[\liminf_{h\to 0} e^{-2\phi_{\min}/h}h^{k-1}\lambda_{k}(h)\geq C_{\rm sup}(k)\,,\]
for the magnetic field $B = |G'(z)|^2\check B\circ G(z)$. This follows from the fact that the function
\[
	y\in \partial D(0,1)\mapsto|(G^{-1})'(y)|\pa_{\n}\phi(G^{-1}(y))
\]
is constant. Here, $\phi$ is defined in \eqref{eq:0modeFlux}.
\end{remark}

Using the Riemann mapping theorem, we can deduce the following lower bound for $\Omega$ with Dini-continuous boundary. Its proof can be found in Section \ref{sec.Riemann}.
\begin{corollary}\label{cor.main}
Assume that $\Omega$ is bounded, simply connected and that $\partial\Omega$ is Dini-continuous. Assume also \eqref{eq.a1}--\eqref{eq.a3} of Theorem \ref{theo.main}. Let $k\in\mathbb{N}^*$. Then, there exist $c_k,C_k>0$ and $h_{0}>0$ such that, for all $h\in(0,h_{0})$,
\[c_kh^{-k+1}e^{2\phi_{\min}/h}\leq\lambda_{k}(h)\leq C_kh^{-k+1}e^{2\phi_{\min}/h}\,.\]
\end{corollary}
\begin{remark}
	Note also that our proof ensures that the constants $C_k,c_k$ can be chosen  so that $C_k/c_k$ does not depend on $k\in \NN^*$.
\end{remark}

Our results can be used to describe the spectrum of the magnetic Laplacian with constant magnetic field (see Remark \ref{rem.nondeg}). 
\begin{corollary}\label{cor.ML}
Assume that $\Omega$ is bounded, strictly convex and that $\partial\Omega$ is Dini-continuous. Assume also that \eqref{eq.a1}--\eqref{eq.a3} of Theorem \ref{theo.main}  hold and that $B$ is constant.

Then, the $k$-th eigenvalue of $(-ih\nabla-A)^2$ with Dirichlet boundary conditions, denoted by $\mu_{k}(h)$, satisfies, for some $c,C>0$ and $h$ small enough,
\begin{equation}\label{eq.B=ct}
B h+ch^{-k+1}e^{2\phi_{\min}/h}\leq \mu_{k}(h)\leq B h+Ch^{-k+1}e^{2\phi_{\min}/h}\,.
\end{equation}
In particular, the first eigenvalues of the magnetic Laplacian are simple in the semiclassical limit. 
\end{corollary}

\subsubsection{Relations with the literature}\label{sec.literature}
Let us compare our result with the existing literature.
\begin{enumerate}[\rm i.]
\item When $B=1$, our results improve the bound obtained by Erd\"os for $\lambda_{1}(h)$ \cite[Theorem 1.1 \& Proposition A.1]{E96} and also the bound by Helffer and Morame \cite[Propositions 4.1 and 4.4]{HelfferMorame}. Indeed, \eqref{eq.B=ct} gives us the optimal behavior of the remainder. When $B=1$ and $\Omega=D(0,1)$, the asymptotic expansion of the next eigenvalues is considered in \cite[Theorem 5.1, c)]{HP17}. Note that, in this case, $\phi=\frac{|x|^2-1}{4}$ and that Theorem \ref{theo.main} allows to recover \cite[Theorem 5.1, c)]{HP17} by considering radial magnetic fields.
\item In \cite{HP17} (simply connected case) and \cite{HP17b} (general case), Helffer and Sundqvist have proved, under assumption \eqref{eq.a1}, that
\[\lim_{h\to 0}h\ln\lambda_{1}(h)=2\phi_{\min}\,.\]
Moreover, under the assumptions \eqref{eq.a1}, \eqref{eq.a2} and \eqref{eq.a3} of Theorem \ref{theo.main}, their theorem \cite[Theorem 4.2]{HP17}, implies the following upper bound for the \emph{first eigenvalue}
\[\lambda_{1}(h)\leq 4\Phi\det(\mathsf{Hess}_{x_{\min}}\phi)^{\frac{1}{2}}(1+o(1))e^{2\phi_{\min}/h}
\,.\]
Note that Theorem \ref{theo.main}  \eqref{theo.upb} provides a better upper bound even for $k=1$.

They also establish the following lower bound by means of rough considerations:
\[\forall h>0\,,\quad\lambda_{1}(h)\geq h^2\lambda^\mathsf{Dir}_{1}(\Omega)e^{2\phi_{\min}/h}\,,\]
where $\lambda_1^{\mathrm{Dir}}(\Omega)$ is the first eigenvalue
of the corresponding magnetic Dirichlet Laplacian. This estimate is itself an improvement of \cite[Theorem 2.1]{EKP16}. 

Corollary \ref{cor.main} is an optimal improvement in terms of the order of magnitude of the pre-factor of the exponential. It also improves the existing results by considering the excited eigenvalues.
Describing the behavior of the prefactor is not a purely technical question. Indeed, it is directly related to the simplicity of the eigenvalues and even governs the asymptotic behavior of the spectral gaps. This simplicity was not known before, except in the case of constant magnetic field on a disk.
\item The problem of estimating the spectrum of the Dirichlet-Pauli operator is closely connected to the spectral analysis of the Witten Laplacian (see for instance \cite[Remark 1.6]{HP17} and the references therein). For example, in this context, the ground state energy is
\begin{equation}\label{eq.W}
\min_{\underset{v\neq 0}{v\in H^1_{0}(\Omega)}}\frac{\int_{\Omega}|h\nabla v|^2e^{-2\phi/h}\dd x}{\int_{\Omega}e^{-2\phi/h}|v|^2\dd x}\,,
\end{equation}
whereas, as known in the literature  in the present paper, we will focus on
\begin{equation}\label{eq.Wbis}
\min_{\underset{v\neq 0}{v\in H^1_{0}(\Omega)}}\frac{\int_{\Omega}|h(\partial_{x_{1}}+i\partial_{x_{2}})v|^2e^{-2\phi/h}\dd x}{\int_{\Omega}e^{-2\phi/h}|v|^2\dd x}\,,
\end{equation}
(see also Lemma \ref{lem.lambdaholo}).
Considering \emph{real-valued} functions $v$ in \eqref{eq.Wbis} reduces to \eqref{eq.W}. In this sense, \eqref{eq.Wbis} gives rise to a \enquote{less elliptic} minimization problem. 
\end{enumerate}

\subsection{About the intuition and strategy of the proof}
In this paragraph we discuss the main lines of our strategy. It is
intended to reveal the intuition behind some of our proofs. We will focus mostly on
the ground-state energy, which is given by \eqref{eq:eigDP} as
\begin{align}
	\label{eq:1}
	\lambda_1(h)=
	\min_{\psi\in H^1_0(\Omega; \CC^2)\setminus\{0\}}		
	\frac{\norm{\sigma\cdot(\p-A)\psi}_{L^2(\Omega)}^2}
	{\norm{\psi}_{L^2(\Omega)}^2}\,.
\end{align}
It is easy to guess from \eqref{eq:defDP} that the ground state energy
has to have the form $\psi=(u,0)^{\rm{T}}$. This is consistent with
the physical intuition that, for low energies, the spin of the
particle should be parallel to the magnetic field.

The variational problem above can be re-written by means of a
suitable transformation as
\begin{align}
	\label{eq:2}
	\lambda_1(h)=h^2 \min_{\underset{v\neq 0}{v\in H^1_{0}(\Omega)}}
	\frac{\int_{\Omega}|2 \partial_{\overline{z}}v|^2e^{-2\phi/h}\dd
		x} {\int_{\Omega} |v|^2 e^{-2\phi/h}\dd x}=: h^2 \min_{\underset{v\neq 0}{v\in H^1_{0}(\Omega)}}
	\frac{\num(v,\phi)}{\den(v,\phi)}  \,,
\end{align}
where, $\partial_{\overline{z}}=(\partial_1+i\partial_2)/2$ and $\phi$
is the unique solution to $\Delta \phi=B$ in $\Omega$ with
Dirichlet boundary conditions (see Definition \ref{defi.gauge}). This connection between the
spectral analysis of the Dirichlet-Pauli and
Cauchy-Riemann operators is known in the literature  (see
e.g. \cite{E96,HP17,BR18} and \cite{T92}), and we
describe it in Section \ref{sec.fund}.

In order to study the problem in \eqref{eq:2} it is helpful to consider
the following heuristics concerning  $\num(v,\phi)$.
\begin{observation}\label{obs1}
	A minimizer $v_h$ {\it wants to be } an analytic function in the
	interior of $\Omega$ but, due to the boundary conditions, has to have
	a different behaviour close to the boundary. So, if we set $\Omega_\delta:=\{x\in\Omega,
	\mathsf{dist}(x,\partial\Omega)\ge \delta\}$ for $\delta>0$, we expect that
	$v_h$ behaves almost as an analytic
	function on $U$ with  $\Omega_\delta \subset U\subset \Omega $. Moreover, this tendency
	is enhanced in the semiclassical limit when the presence of the
	magnetic field becomes stronger. Hence, we also expect that $\delta\to 0$
	as $h\to 0$ in some way.
\end{observation}
We comment below how we make Observation \ref{obs1} more precise,
for the moment let us just mention that throughout this discussion we work
with $\delta$ such that
\begin{align}
	\label{eq:5}
	\delta^2/h\to 0\quad \mbox{and} \quad\delta/h\to \infty \quad\mbox{as}\quad h\to 0\,.
\end{align}
As a  consequence of  Observation \ref{obs1} we expect that 
\begin{align}
	\label{eq:3}
	\num(v_h,\phi)\sim \int_{T_\delta} |2 \partial_{\overline{z}}v_h|^2e^{-2\phi/h}\dd
	x\,,
\end{align}
where $T_\delta:=\Omega\setminus \Omega_\delta$.

An essential ingredient in our method  is the analysis of 
the minimization problem associated with the RHS of \eqref{eq:3}.
The main ideas go as
follows: Assume first that $\Omega$
is the disk $D(0,1)$.  By writing the integrand $
|\partial_{\overline{z}}v_h|^2e^{-2\phi/h}$ in tubular coordinates
(see Item i. from the proof of Lemma \ref{lem:estiannularUpper} )  and Taylor
expanding $\phi$ around any point at the boundary $\partial\Omega$ we
get,
for $\delta$ satisfying \eqref{eq:5},
\begin{align}
	\label{eq:6}
	\int_{T_\delta} |2 \partial_{\overline{z}}v|^2e^{-2\phi/h}\dd x
	&=(1+o(h))   \int_0^{2\pi} \int_0^\delta e^{2t \pa_\n \phi
		/h}|(\partial_\tau-i\partial_s)v|^2 \dd s \dd \tau\\
	&=:(1+o(h))  J_h(v)
\end{align}
(see also the proof of Lemma \ref{lem:lowbound_annular_energy}), where
$\pa_\n \phi \equiv \pa_\n \phi(s)$ is the normal derivative at the
boundary (see Notation \ref{not.BH}).

Observe that if $\pa_\n \phi$ is a
constant along the boundary, then it equals the flux $\Phi$. In this
case,  as explained in Item iv. of the  proof of
Lemma \ref{lem:lowbound_annular_energy}, the problem of finding
a non-trivial solution of 
\begin{align}
	\label{eq:9}
	\inf_{v\in H^1(T_\delta)} J_h(v)\,\quad \mbox{with}\quad v\!\upharpoonright_{\partial
		\Omega_\delta}=v_\delta\,, v\!\upharpoonright_{\partial
		\Omega}=0\,,
\end{align}
can be reduced to a sum (labeled in the Fourier index) of one-dimensional problems that we solve explicitly in Lemma
\ref{lem.opt}.

For the particular case of $v$ having  only
the non-negative Fourier modes on $\partial \Omega_\delta$ (i.e.,
$v_\delta=\sum_{m\ge 0} \hat{v}_{\delta,m}
e^{ims}$) we find that (see Lemma \ref{lem:lowbound_annular_energy}) 
\begin{align}
	\label{eq:8}
	J_h(v)\ge \frac{\Phi/h}{1-e^{-2\delta\Phi/h}} \norm{v}_{L^2(\partial
		D(0, 1-\delta))}^2=(1+o(h)) 2\Phi/h \norm{v}_{L^2(\partial
		D(0, 1-\delta))}^2\, 
\end{align}
where the last equality is a trivial consequence of
\eqref{eq:5}. Moreover, by Lemma \ref{lem.opt}, the latter inequality
is saturated when $v_\delta= \hat{v}_{\delta,0}$. Concerning, the assumption on $v$, recall that
analytic functions on the disk have only Fourier modes for $m\ge 0$.

Notice that if $B$ is rotationally symmetric  $\pa_\n
\phi$ is constant. If $\pa_\n
\phi$ is not a constant we can give  a suitable estimate using that
$\min_{\partial \Omega} \pa_\n
\phi>0$. We extend the previous analysis to more general geometries by
using the Riemann mapping theorem.

There is another important point to take into account, this
time concerning $\den(v,\phi)$.
\begin{observation}\label{obs2}
	Recall that $\phi\le 0$ has an absolute, non-degenerate, minimum at $x_{\rm
		min}$. Hence, the weighted norm of $v_h$, $\den(v,\phi)$, should have a tendency to
	concentrate  around $x_{\rm min}$. This is made precise in Lemma
	\ref{lem.normL2} below.  Moreover, observe that using Laplace's method, one
	formally gets that, as $h\to 0$,
	\begin{align}
		\label{eq:4}
		\den(v,\phi)\sim  h \pi |v(x_{\rm min})|^2 e^{-2\phi_{\min}/h} (\det \mathsf{Hess}_{x_{\min}}\phi)^{-\frac{1}{2}}\,.
	\end{align}
\end{observation}
Observations \ref{obs1} and \ref{obs2} reveal the importance of the behaviour of a
minimizer around the boundary and close to $x_{\rm{min}}$,
respectively. In addition, this behaviour is naturally captured
through the norms $\Nh$ and $\Nb$ given in Definition \ref{not.BH},
which, in turn, provide a natural Hilbert space structure to select
linear independent test functions which are used to estimate the
excited energies.

In order to show our result we give  upper and lower  bound to the
variational problem \eqref{eq:2}. This is done in Sections \ref{sec.3}
and \ref{sec.4}, respectively. Concerning the upper bound: In view of
the previous discussion it is natural to choose a trial function (at
least for the disk, see Remark \ref{rem.naive})
$v=\omega \chi$ where $\omega$ is an analytic function in $\Omega$ and
$\chi$ is such that $\chi\!\upharpoonright_{\Omega_\delta}=1$ and
decays smoothly to zero towards $\partial\Omega$. We pick
$\chi\!\upharpoonright_{T_\delta}$ as an optimizer of the problem
\eqref{eq:9}. For $\lambda_k(h)$, we choose $\omega$ to be a
polynomial of degree $(k-1)$. In particular, for the ground-state
energy, $\omega$ is constant and in view of \eqref{eq:8} and
\eqref{eq:4} we readily see how the claimed upper bound (at least for
the disk with radial magnetic field) is obtained.

As for the lower bound, as a preliminary step, we discuss  in Section \ref{sec.magnCR} 
some ellipticity properties related to the magnetic Cauchy-Riemann
operators. Our main result there is Theorem
\ref{prop.elliptic}. 
It provides elliptic estimates for the magnetic Cauchy-Riemann
operators on the orthogonal of the kernel which consists, up to an
exponential weight, in holomorphic functions. The findings of Section
\ref{sec.magnCR} are crucial to prove  Proposition
\ref{prop.approxH1norm}, which gives estimates on the behaviour
described in Observation
\ref{obs1}. Indeed, Proposition
\ref{prop.approxH1norm}, together with the upper bound, roughly
states that the non-analytic part of $v_h$ on any open set contained
in $\Omega$ is, in the semiclassical limit, exponentially small in a
sufficiently strong norm.  At least for the disk with
radial magnetic field, we can argue on how to get the lower bound
if we assume  $v_h$ to be  analytic on an open set $U$ with
$ \overline{D(0, 1-\delta)}\subset U \subset D(0, 1)$. Notice that \eqref{eq:8} holds. Moreover, by Cauchy's
Theorem we have $2\pi |v_h(x_{\rm min})|^2= 2\pi |v_h(0)|^2\le
(1+o(h))\|v_h\|^2_{\partial D(0, 1-\delta)}$. In this way we see that
the lower bound appears by
combining \eqref{eq:8} and \eqref{eq:4}.

Let us finally remark that actually, since the function $v$ in
\eqref{eq:4} depends on $h$, Laplace's method cannot be performed so
easily. Instead, after the change of scale
$y = \frac{x-x_{\rm min}}{h^{1/2}}$, one has formally the Bargmann
norm appearing:
\begin{align}
	\label{eq:7}
	\den(v,\phi)\sim h  e^{-2\phi_{\min}/h} \int |v(x_{\rm min}+h^{1/2}y)|^2e^{-\mathsf{Hess}_{x_{\min}}\phi (y,y)} \dd y\,.
\end{align}
Ultimately, in the case of the disk with radial magnetic field, Problem \ref{eq:2} reduces formally to
\begin{align}
	\label{eq:8}
	\lambda_1(h) \gtrsim e^{2\phi_{\min}/h}  \inf_{\underset{v\neq 0}{v\in \mathscr{H}(\Omega)}} 2\left(\frac{ \Nh(v)}{\Nb(v(x_{\rm min}+h^{1/2}\cdot))}\right)^2\,,
\end{align}
which can be computed easily due to the orthogonality of the polynomials $(z^{n})_{n\geq0}$ in the Hilbert spaces $\mathscr{H}^2(\Omega)$ and $\mathscr{B}^2(\CC)$ (see Remark \ref{rem.sym}).
Of course, special attention has to be paid on the domains of integration and  the sets where the holomorphic tests functions live.
In the non-radial case however, we strongly use the multi-scale structure of \eqref{eq:8} to get the result of Theorem \ref{theo.main} (see Section \ref{sec.final}). Note that the constant $\theta_0$ of Theorem \ref{theo.main} which appears in the computation of \eqref{eq:8} somehow measures a symmetry breaking rate (see Remark \ref{rem.theta0} and Lemma \ref{lem:lowbound_4}).

	\section{Change of gauge}\label{sec.fund}
		The following result allows to remove the magnetic field up to sandwiching the Dirac operator with a suitable matrix.
		\begin{proposition}\label{prop:zeromode}
			We have 
			\be\label{eq:facto}
				e^{\sigma_3\phi/h}\sigma\cdot \p e^{\sigma_3\phi/h} = \sigma\cdot(\p-A)\,,
			\ee
			 as operators acting on $H^1(\Omega, \CC^2)$ functions.
		\end{proposition}
		%
		%
The proof follows from the next two lemmas and Definition \ref{defi.gauge} (see also \cite[Theorem 7.3]{T92}).
		\begin{lemma}\label{lem.ant-com}
		Let $f : \CC\to\CC$ be an entire function and $A,B$ be two square matrices such that $AB=-BA$. Then,
		\[Af(B)=f(-B)A\,.\]
		\end{lemma}

		\begin{lemma}[Change of gauge for the Dirac operator]\label{lem.gauge}
		Let $\Phi:\Omega\rightarrow \RR$ be a regular function. We have
			\[
				e^{\sigma_3\Phi}\(\sigma\cdot \p\)e^{\sigma_3\Phi}=\sigma\cdot(\p-h\nabla\Phi^\perp)\,
			\]
		 as operators acting on $H^1(\Omega, \CC^2)$ functions  and where $\nabla\Phi^\perp$ is defined in Definition \ref{defi.gauge}. 
		\end{lemma}
		\begin{proof}
			By Lemma \ref{lem.ant-com}, we have for $k=1,2$ that
				\[
					e^{\sigma_3\Phi}\sigma_k =\sigma_k e^{-\sigma_3\Phi}\,.
				\]
			Thus, by the Leibniz rule,
			\[e^{\sigma_3\Phi}\(\sigma\cdot \p\)e^{\sigma_3\Phi}=(\sigma e^{-\sigma_{3}\Phi}\cdot \p)e^{\sigma_3\Phi}=\sigma\cdot\left(\p-ih\sigma_{3}\nabla\Phi\right)\,.\]
			It remains to notice that $-i\sigma \sigma_{3}=\sigma^\perp:=(-\sigma_{2}, \sigma_{1})$ so that
			\[e^{\sigma_3\Phi}\(\sigma\cdot \p\)e^{\sigma_3\Phi}=\sigma\cdot \p+h\sigma^\perp\cdot\nabla\Phi=\sigma\cdot \p-h\sigma\cdot\nabla\Phi^\perp\,.\]

		\end{proof}
	We let
	\[\partial_{z} := \frac{\partial_x-i\partial_y}{2}\,,\qquad\partial_{\overline{z}} := \frac{\partial_x+i\partial_y}{2}\,.\] 
	We obtain then the following result.
	\begin{lemma}\label{lem.lambdaholo}
		Let $k\in \NN^*$ be such that $\lambda_k(h)< 2B_0 h$. Then, we have
	\begin{equation}\label{eq.lambda1'}
		\lambda_k(h) =  \inf_{
					\begin{array}{c}
						V\subset H^1_0(\Omega;\CC)\,,\\
						\dim V = k\,,
					\end{array}
					}
					\sup_{
					\begin{array}{c}
						v\in V\setminus\{0\}\,,
					\end{array}
					}
					\frac{4\int_{\Omega}e^{-2\phi/h}|h\partial_{\overline{z}}v|^2\dd x}{\int_{\Omega}|v|^2e^{-2\phi/h}\dd x}\,.
	\end{equation}
		We recall that $\lambda_k(h)$ is defined in \eqref{eq:eigDP}.
	\end{lemma}
	\begin{proof}
	By \eqref{eq:defDP} and \eqref{eq:eigDP}, we get since $\mathscr{L}^+_{h}\geq 2B_0h$,
	\[
		\lambda_k(h) = \inf_{
				\begin{array}{c}
					V\subset H^1_0(\Omega;\CC)\,,\\
					\dim V = k\,,
				\end{array}
				}
				\sup_{
				\begin{array}{c}
					u\in V\setminus\{0\}\,,
				\end{array}
				}
				\frac{\norm{\sigma\cdot(\p-A)\left(\begin{array}{c}u\\0\end{array}\right)}_{L^2(\Omega)}^2}{\norm{u}_{L^2(\Omega)}^2}\,.
	\]
	Let $u\in H^1_0(\Omega;\CC)$ and $h>0$. Letting $u=e^{-\phi/h}v$, we have, by Proposition \ref{prop:zeromode},
\[
	\norm{\sigma\cdot(\p-A)\left(\begin{array}{c}u\\0\end{array}\right)}_{L^2(\Omega)}^2 = \norm{e^{\sigma_3\phi/h}\sigma\cdot \p \begin{pmatrix}v\\0\end{pmatrix}}_{L^2(\Omega)}^2 = 4\int_{\Omega}e^{-2\phi/h}|h\partial_{\overline{z}}v|^2\dd x\,,
\]
and
\[
	\norm{u}_{L^2(\Omega)}^2 = \int_{\Omega}|v|^2e^{-2\phi/h}\dd x\,.
\]
	\end{proof}

\section{Upper bounds}\label{sec.3}
This section is devoted to the proof of the following upper bounds.

\begin{proposition}\label{prop:upper-bound}
Assume that $\Omega$ is $\mathscr{C}^2$ and satisfies {\rm Assumption \ref{asum:setom}}. For all $k\in\mathbb{N}^*$, we have
\begin{equation}\label{eq.ubin1}
\lambda_{k}(h)\leq  C_{\rm sup}(k) h^{-k+1}e^{2\phi_{\min}/h}(1+o(1))\,,
\end{equation}
where $\lambda_{k}(h)$ and $C_{\rm sup}(k)$ are defined in \eqref{eq:eigDP} and in {\rm Theorem \ref{theo.main}} respectively.
\end{proposition}

\subsection{Choice of test functions}\label{sec:upbdheuristic}
Let $k\in\NN^*$ and $m\in\NN$. 
By Formula \eqref{eq.lambda1'}, we look for a $k$-dimensional subspace $V_{h}$ of  $H^1_0(\Omega; \CC)$ such that 
\[
\sup_{v\in V_h\setminus\{0\}}\ \frac{4h^2\int_{\Omega}|\partial_{\overline{z}}v|^2e^{-2\phi/h}\dd x}{\int_{\Omega}|v|^2e^{-2(\phi-\phi_{\min})/h}\dd x}\,\leq C_{\rm sup}(k) h^{-k+1}(1+o(1))\,.
\]
Using the min-max principle, this would give \eqref{eq.ubin1}.
Formula \eqref{eq.lambda1'} suggests to take functions of the form
\[
	v(x)=\chi(x)w(x)\,,
\]
where
\begin{enumerate}[\rm i.]
\item $w$ is holomorphic on a neighborhood on $\Omega$,
\item the function $\chi: \overline{\Omega}\to [0,1]$ is a Lipschitzian function satisfying the Dirichlet boundary condition and being $1$ away from a fixed neighborhood of the boundary.
\end{enumerate}
In particular, there exists $\ell_0\in(0,\mathsf{d}(x_{\min},\pa\Omega))$ such that 
\begin{equation}\label{eq:defl0}
	\chi(x) = 1, \quad \mbox{ for all }x\in \Omega \mbox{ such that } \mathsf{d}(x,\pa\Omega)>\ell_0\,,
\end{equation}
where $\mathsf{d}$ is the usual Euclidean distance. 
\begin{remark}\label{rem.naive}
The most naive test functions set could be the following
\[
	V_h = {\rm span}(\chi_h(z),\dots, \chi_h(x)(z-z_{\rm min})^{k-1})\,,
\]
where $(\chi_h)_{h\in(0,1]}$ satisfy \eqref{eq:defl0}.
With this choice, one would get
\[
\sup_{v\in V_h\setminus\{0\}}\ \frac{4h^2\int_{\Omega}|\partial_{\overline{z}}v|^2e^{-2\phi/h}\dd x}{\int_{\Omega}|v|^2e^{-2(\phi-\phi_{\min})/h}\dd x}\,\leq \widetilde{C_{\rm sup}}(k) h^{-k+1}(1+o(1))\,,
\]
where
\[
	\widetilde{C_{\rm sup}}(k) = 2\left(\frac{\Nh \left((z-z_{\rm min})^{k-1}\right)}{\distb \left(z^{k-1},\mathcal{P}_{k-2}\right)}\right)^2\geq C_{\rm sup}(k) \,.
\]
Note however that in the radial case $\widetilde{C_{\rm sup}}(k) = C_{\rm sup}(k)$. 
We will rather use functions compatible with the Hardy space structure to get the bound of Proposition \ref{prop:upper-bound}, as explained below.
\end{remark}
\begin{notation}\label{not.Pn}
Let us denote by $(P_n)_{n\in\NN}$ the $\Nb$-orthogonal family  such that $P_n(Z) = Z^n + \sum_{j=0}^{n-1}b_{n,j}Z^j$ obtained after a Gram-Schmidt process on $(1,Z,\dots, Z^n,\dots)$. Since $P_n$ is $\Nb$-orthogonal to $\mathcal{P}_{n-1}$,  we have
\begin{equation}\label{rem:testfuncupbound}
	\begin{split}
	\distb\left(Z^{n},\mathcal{P}_{n-1}\right)
	&=\distb\left(P_n,\mathcal{P}_{n-1}\right)  
	= \inf\{\Nb(P_n-Q)\,, Q\in\mathcal{P}_{n-1}\}
	\\&
	=\inf\{\sqrt{\Nb(P_n)^2+\Nb(Q)^2}\,, Q\in\mathcal{P}_{n-1}\}
	= \Nb(P_{n})\,,\mbox{ for }n\in\NN\,.
	\end{split}
\end{equation}
%
%
Let $Q_{n}\in \mathscr{H}^{2}_k(\Omega)$ be the unique function such that
\[\disth \left((z-z_{\rm min})^{n},\mathscr{H}^2_k(\Omega)\right) = \Nh((z-z_{\rm min})^n-Q_n(z))\,,\]
for $n\in \{0,\dots,k-1\}$, (see Remark \ref{rem.hilbertprojhardy}). We recall that $\Nb$, $\Nh$, $\mathcal{P}_{n-1}$, and $\mathscr{H}^2_k(\Omega)$ are defined in Section \ref{sec.mt}.
\end{notation}
\begin{lemma}\label{lem:hardy_approx}
	For all $n\in \{0,\dots,k-1\}$, there exists a sequence $(Q_{n,m})_{m\in \NN}\subset \mathscr{H}^2_k(\Omega)\cap W^{1,\infty}(\Omega)$ that converges to $Q_n$ in $\mathscr{H}^2(\Omega)$.
\end{lemma}
\begin{proof}
	We can write $Q_n(z) = (z-z_{\rm min})^{k-1}\widetilde Q_n(z)$. Here, $\widetilde Q_n$ is an holomorphic function on $\Omega$. Since $z\mapsto (z-z_{\rm min})^{1-k}\in L^{\infty}(\partial \Omega)$, we get $\widetilde Q_n\in\mathscr{H}^2(\Omega)$. By Lemma \ref{lem.density}, there exists a sequence $(\widetilde Q_{n,m})_{m\in \NN}\subset\mathscr{H}^2(\Omega)\cap W^{1,\infty}(\Omega)$ converging to $\widetilde Q_n$ in $\mathscr{H}^2(\Omega)$. We have
	\[
		\Nh((z-z_{\rm min})^{k-1}(\widetilde Q_{n,m}-\widetilde Q_{n}))\leq \norm{(z-z_{\rm min})^{k-1}}_{L^\infty(\partial \Omega)}\Nh(\widetilde Q_{n,m}-\widetilde Q_{n})\,,
	\]
	so that the sequence $(Q_{n,m})_{m\in \NN} = ((z-z_{\rm min})^{k-1}\widetilde Q_{n,m})_{m\in \NN}\subset \mathscr{H}^2_k(\Omega)$ converges to $Q_n$ in $\mathscr{H}^2(\Omega)$. Since $z\mapsto (z-z_{\rm min})^{k-1}\in L^{\infty}(\partial \Omega)$, $Q_{n,m}\in \mathscr{H}^2_k(\Omega)$.
\end{proof}
Let us now define the $k$-dimensional vector space $V_{h,k,\rm sup}$ by
\begin{equation}\label{eq.spaceupbd}
	V_{h,k,\rm sup} = {\rm span}(w_{0,h},\dots,w_{k-1,h})\,,
\end{equation}
\[
	w_{n,h}(z) = h^{-\frac{1}{2}}P_{n}\left(\frac{z-z_{\rm min}}{h^{1/2}}\right) - h^{-\frac{1+n}{2}}Q_{n,m}(z),\mbox{ for }n\in \{0,\dots,k-1\}\,.
\]

At the end of the proof, $m$ will be sent to $+\infty$. Note that we will not need the uniformity of the semiclassical estimates with respect to $m$. That is why the parameter $m$ does not appear in our notations. Note that $w_{n,h}$, being a non trivial holomorphic function, does not vanish identically at the boundary. To fulfill the Dirichlet condition, we have to add a cutoff function (see below).

\begin{remark}\label{rem.hardy_approx}
	Consider
	\[
	\widetilde\omega_{n,h}(z) = h^{-\frac{1}{2}}P_{n}\left(\frac{z-z_{\rm min}}{h^{1/2}}\right) - h^{-\frac{1+n}{2}}Q_{n}(z)\,.
	\]
	Since $Q_n$ belongs to $\mathscr{H}^2(\Omega)\not\subset H^1(\Omega;\CC)$, the functions $\widetilde w_{n,h}:x\mapsto\widetilde \omega_{n,h}(x_1+ix_2)$ and $\chi \widetilde w_{n,h}$ do not belong necessarily to $H^1(\Omega;\CC)$ and $H^1_0(\Omega;\CC)$ respectively. 
	That is why we introduced $Q_{n,m}$. 
	Note that to get $H^1_0(\Omega; \CC)$ test functions, it suffices to impose that $\chi$ is compactly supported in $\Omega$. With this strategy, our proof can be adapted to the case where $\Omega$ is non necessarily simply connected.

\end{remark}

\subsection{Estimate of the $L^2$-norm}
The aim of this section is to prove the following estimate.

\begin{lemma}\label{lem:L2esti1}
Let  $h\in(0,1]$, $v_h = \chi\sum_{j=0}^{k-1}c_j w_{j,h}$ with $c_0,\dots c_{k-1}\in \CC$, $\chi$ satisfying \eqref{eq:defl0} and $(w_{j,h})_{j\in\{0,\dots,k-1\}}$ defined in \eqref{eq.spaceupbd}. We have 
\begin{equation}\label{eq:normuV2}
		\int_{\Omega}|v_h|^2 e^{-2(\phi(x)-\phi_{\rm min})/h}\dd x
		 = (1+o(1))\sum_{j=0}^{k-1}|c_j|^2\Nb(P_j)^2\,,
\end{equation}
where $\Nb$ is defined in Notation \ref{not.BH} and $o(1)$ does not depend on $c = (c_0,\dots,c_{k-1})$ and $\chi$.
\end{lemma}

\begin{proof}
	Let $\alpha\in\left(\frac{1}{3},\frac{1}{2}\right)$, $n,n'\in \{0,\dots, k-1\}$. 
 
		In the proof, three types of terms will appear after a change of scale around $x_{\rm min}$ : $\braket{P_n, P_{n'}}_{\mathcal{B}}$, $\braket{P_n, Q_{n',m}}_{\mathcal{B}}$ and $\braket{Q_{n,m}, Q_{n',m}}_{\mathcal{B}}$  where $\langle\cdot,\cdot\rangle_{\mathcal{B}}$ is the scalar product associated with $\Nb$. Since the polynomials $(P_n)_{n\in\NN}$ are $\Nb$-orthogonal, we have $\braket{P_n, P_{n'}}_{\mathcal{B}} = 0$ if $n\ne n'$ and we will prove that $\braket{Q_{n,m}, Q_{n',m}}_{\mathcal{B}} = \mathcal{O}(h)$ and by Cauchy-Schwarz inequality $\braket{P_n, Q_{n',m}}_{\mathcal{B}} = \mathcal{O}(h^{1/2})$. More precisely, we have: 
	\begin{enumerate}[\rm i.]
		\item \label{st.estil2_1}
		Let us estimate the weighted scalar products related to $P_{n}$ for the weighted $L^2$-norm.
		
	Using the Taylor expansion of $\phi$ at $x_{\rm min}$, we get, for all $x\in D(x_{\rm min},h^\alpha)$,	\begin{equation}\label{eq.tayint}
		\frac{\phi(x)-\phi_{\rm min}}{h} =  \frac{1}{2h}\mathsf{Hess}_{x_{\min}}\phi(x-x_{\rm min},x-x_{\rm min}) + \mathscr{O}(h^{3\alpha-1})\,.
	\end{equation}
	By using the change of coordinates
\begin{equation}\label{eq.Lambda}
	\Lambda_h : x\longmapsto \frac{x-x_{\rm min}}{h^{1/2}}\,,
\end{equation}
	we find
	\begin{equation}\label{eq.norml2_1}\begin{split}
		&\int_{D(x_{\rm min},h^\alpha)}h^{-1}P_{n}P_{n'}\left(\frac{x_1+ix_2-z_{\rm min}}{h^{1/2}}\right) e^{-2(\phi(x)-\phi_{\rm min})/h}\dd x
		\\&\quad
		= (1 + \mathscr{O}(h^{3\alpha-1}))\int_{D(x_{\rm min},h^\alpha)}h^{-1}P_{n}P_{n'}\left(\frac{x_1+ix_2-z_{\rm min}}{h^{1/2}}\right) e^{-\frac{1}{h}\mathsf{Hess}_{x_{\min}}\phi(x-x_{\rm min},x-x_{\rm min})}\dd x
		\\&\quad
		= (1 + \mathscr{O}(h^{3\alpha-1}))\int_{D(0,h^{\alpha-1/2})}P_{n}P_{n'}\left(y\right) e^{-\mathsf{Hess}_{x_{\min}}\phi(y,y)}\dd y
		\\&\quad
		= (1 + \mathscr{O}(h^{3\alpha-1}))\left(\langle P_n, P_{n'}\rangle_{\mathcal{B}}-\int_{\CC\setminus D(0,h^{\alpha-1/2})}P_{n}P_{n'}\left(y\right) e^{-\mathsf{Hess}_{x_{\min}}\phi(y,y)}\dd y\right)
		\\&\quad
		= (1 + \mathscr{O}(h^{3\alpha-1}))\langle P_n, P_{n'}\rangle_{\mathcal{B}}+\mathscr{O}(h^\infty)\,,
	\end{split}\end{equation}
	where the last equality follows from Assumption \eqref{eq.a3} in Theorem \ref{theo.main}.

	We recall Assumptions \eqref{eq.a2} and \eqref{eq.a3} of Theorem \ref{theo.main}. Then, by the Taylor expansion of $\phi$ at $x_{\rm min}$, we deduce that
	\begin{equation}\label{eq.tayext}
		\inf_{\Omega\setminus D(x_{\rm min},\ h^{\alpha})}\phi\geq\phi_{\rm min} + \frac{\lambda_{\rm min}}{2}h^{2\alpha}(1+\mathscr{O}(h^\alpha))\,,
	\end{equation}
	where $\lambda_{\rm min}>0$ is the lowest eigenvalue of $\mathsf{Hess}_{x_{\min}}\phi $.
	Since $P_n$ is of degree $n$, there exists $C>0$ such that
	\[
		\sup_{x\in\Omega}\left|h^{-\frac{1}{2}}P_{n}\left(\frac{x_1+ix_2-z_{\rm min}}{h^{1/2}}\right)\right|\leq Ch^{-\frac{n+1}{2}}\,.
	\]
	Using this with \eqref{eq.tayext}, we get
	\begin{multline}\label{eq.norml2_2}
\left|\int_{\Omega \setminus D(x_{\rm min},h^\alpha)}h^{-1}\chi^2 P_{n}P_{n'}\left(\frac{x_1+ix_2-z_{\rm min}}{h^{1/2}}\right) e^{-2(\phi(x)-\phi_{\rm min})/h}\dd x\right|\\
\leq Ch^{-\frac{n+1}{2}}h^{-\frac{n'+1}{2}}e^{-\lambda_{\rm min}h^{2\alpha-1}(1+ \mathscr{O}(h^\alpha))} = \mathscr{O}(h^\infty)\,.
	\end{multline}
	From \eqref{eq.norml2_1} and \eqref{eq.norml2_2}, we find
	\begin{multline}\label{eq.Pn}
	\int_{\Omega}h^{-1}\chi^2 P_{n}P_{n'}\left(\frac{x_1+ix_2-z_{\rm min}}{h^{1/2}}\right) e^{-2(\phi(x)-\phi_{\rm min})/h}\dd x\\
	=(1 + \mathscr{O}(h^{3\alpha-1}))\langle P_n, P_{n'}\rangle_{\mathcal{B}}+\mathscr{O}(h^\infty)\,.
	\end{multline}
	
		\item Let us now deal with the weighted scalar products related to the $Q_{n,m}$. Let $u\in \mathscr{H}^2(\Omega)$ and $z_0\in D(z_{\rm min},h^\alpha)$. 
	By the Cauchy formula (see \cite[Theorem $10.4$]{duren2000theory}) and the Cauchy-Schwarz inequality,
	\begin{equation}\label{eq.CCS}\begin{split}
		|u^{(k)}(z_0)|
		&= \frac{k!}{2\pi}\left|
			\int_{\partial \Omega}\frac{u(z)}{(z-z_0)^{k+1}}\dd z
		\right|
		\\&
		\leq \frac{k!}{2\pi\sqrt{\min_{\partial \Omega}\pa_\n \phi}}\Nh(u)\left(
			\int_{\partial\Omega}\frac{|\dd z|}{|z-z_0|^{2(k+1)}}
		\right)^{1/2}
		\\&
		\leq \frac{k!}{2\pi\sqrt{\min_{\partial \Omega}\pa_\n \phi}}\Nh(u)\left(
			\int_{\partial\Omega}\frac{|\dd z|}{(|z-z_{\rm min}|-h^\alpha)^{2(k+1)}}
		\right)^{1/2}\leq C\Nh(u)\,.
	\end{split}\end{equation}
	With the Taylor formula for $u=Q_{n,m}$ at $z_{\rm min}$, this gives
	\[
		|Q_{n,m}(z_0)|\leq C|z_0-z_{\rm min}|^k\Nh(Q_{n,m})\,.
	\]
	Using \eqref{eq.tayint}, this implies
	\begin{equation}\label{eq.norml2_3}\begin{split}
		&\int_{D(x_{\rm min},h^\alpha)}|h^{-\frac{1+n}{2}}Q_{n,m}(x_1+ix_2)|^2 e^{-2(\phi(x)-\phi_{\rm min})/h}\dd x
		\\&\quad
		\leq Ch^{-(1+n)}\int_{D(x_{\rm min},h^\alpha)}|(x_1+ix_2)-z_{\rm min}|^{2k} e^{-2(\phi(x)-\phi_{\rm min})/h}\dd x
		\\&\quad
		\leq Ch^{k-n}\Nb(z^k)^2\leq Ch\,.
	\end{split}\end{equation}		
	Using \eqref{eq.tayext} and $Q_{n,m}\in W^{1,\infty}(\Omega)\subset L^2(\Omega)$, we get
	\begin{equation}\label{eq.norml2_4}\begin{split}
		&\int_{\Omega \setminus D(x_{\rm min},h^\alpha)}|h^{-\frac{1+n}{2}}\chi Q_{n,m}(x_1+ix_2)|^2 e^{-2(\phi(x)-\phi_{\rm min})/h}\dd x
		\\&\quad
		\leq Ch^{-(n+1)}\norm{Q_{n,m}}_{L^2(\Omega)}^2e^{-\lambda_{\rm min}h^{2\alpha-1}(1+\mathscr{O}(h^\alpha))} = \mathscr{O}(h^\infty)\,.
	\end{split}\end{equation}	
	With \eqref{eq.norml2_3} and \eqref{eq.norml2_4}, we deduce
	\begin{equation}\label{eq.Qnm}
	\int_{\Omega}|h^{-\frac{1+n}{2}}\chi Q_{n,m}(x_1+ix_2)|^2 e^{-2(\phi(x)-\phi_{\rm min})/h}\dd x=\mathscr{O}(h)\,.
	\end{equation}
	
	With the Cauchy-Schwarz inequality, and \eqref{eq.Qnm},
	\begin{equation}\label{eq.norml2_8}
\int_{\Omega}\chi^2 h^{-\frac{1+n}{2}}Q_{n,m}(x_1+ix_2)) h^{-\frac{1+n'}{2}}\overline{Q_{n',m}(x_1+ix_2)}e^{-2(\phi(x)-\phi_{\rm min})/h}\dd x=\mathscr{O}(h)\,.
\end{equation}

		\item Let us now consider the scalar products involving the $P_{n}$ and the $Q_{n',m}$.
	Using \eqref{eq.Pn}, \eqref{eq.norml2_8}, and the Cauchy-Schwarz inequality, we get
	\begin{equation}\label{eq.norml2_7}
		\int_{\Omega}\chi^2 h^{-\frac{1}{2}}P_{n}\left(\frac{x_1+ix_2-z_{\rm min}}{h^{1/2}}\right) h^{-\frac{1+n'}{2}}\overline{Q_{n',m}(x_1+ix_2)} e^{-2(\phi(x)-\phi_{\rm min})/h}\dd x=\mathscr{O}(h^{1/2})\,.
\end{equation}
	\end{enumerate}
		The conclusion follows by expanding the square in the left-hand-side of \eqref{eq:normuV2} and by using \eqref{eq.Pn}, \eqref{eq.norml2_8}, \eqref{eq.norml2_7} .
\end{proof}
\begin{remark}\label{rem.free}
From Lemma \ref{lem:L2esti1}, we deduce that the vectors $\{\chi w_{j,h}\,,0\leq j\leq k-1\}$ are linearly independent for $h$ small enough.
\end{remark}

\subsection{Estimate of the energy}
The aim of this section is to bound from above the energy on an appropriate subspace.

\begin{lemma}\label{lem:estiannularUpper}
	There exists a family of functions $(\chi_h)_{h\in(0,1]}$ which satisfy \eqref{eq:defl0} and such that, for all $w_h = \sum_{j=0}^{k-1}c_j w_{j,h}\in V_{h,k,\rm sup}$ with $c_0,\dots c_{k-1}\in \CC$, 
	\begin{multline*}
	4\int_{\Omega}h^2e^{-2\phi/h} |\partial_{\overline{z}}(\chi_{h}w_h)|^2\dd x\\
	 \leq 2h^{1-k}|c_{k-1}|^2\Nh\left((z-z_{\min})^{k-1}-Q_{k-1,m}\right)+o(1)h^{1-k}\|c\|^2_{\ell^2}\,.
	\end{multline*}
	Here, $o(1)$ does not depend on $c_0,\dots c_{k-1}$.
\end{lemma}
\begin{proof}
Let $\chi$ be any function satisfying \eqref{eq:defl0}.
We have 
\[
4\int_{\Omega}h^2e^{-2\phi/h} |\partial_{\overline{z}}\chi w_h|^2\dd x=h^2\int_{\Omega} |w_h|^2e^{-2\phi/h}|\nabla\chi|^2 \dd x=h^2\int_{\mathsf{supp \nabla\chi}} |w_h|^2e^{-2\phi/h}|\nabla\chi|^2 \dd x\,,
\]
where we have used that $|\nabla\chi|^2=4|\partial_{\overline{z}}\chi|^2$ since $\chi$ is real  and $\partial_{\overline z}w_h = 0$. 

The proof is now divided into three steps. First, we introduce tubular coordinates near the boundary, then we make an explicit choice of $\chi$, and finally, we control the remainders.

\begin{enumerate}[\rm i.]
	\item
	We only need to define $\chi$ in a neighborhood of $\Gamma=\partial\Omega$. To do this, we use the tubular coordinates given by the map
	\[
 		\eta: 	
 			\begin{array}{l}
 				\RR\slash\left( |\Gamma|\ZZ \right)\times (0, t_0) \to\Omega \\
 				(s,t) \mapsto \gamma(s) - t\n(s)	
			 \end{array}
	\]
	for $t_0$ small enough, $\gamma$ being a parametrization of $\Gamma$ with $|\gamma'(s)|=1$ for all $s$, and $\n (s)$ the unit outward pointing normal at point $\gamma(s)$ (see e.g. \cite[\S F]{FH11}). We let
	\[
		\eta^{-1}(x) = (s(x),t(x))\,, \mbox{ for all }x\in \eta\left(\RR\slash\left( |\Gamma|\ZZ \right)\times (0, t_0)\right)\,,
	\] 
	the inverse map to $\eta$.
	We let, for all $x\in \Omega$,
	\[
		\chi(x) = 
		\begin{cases}
			\rho\left(s(x),\mathsf{d}(x,\partial\Omega)\right)&\mbox{ if }\mathsf{d}(x,\partial\Omega)\leq\eps\,,\\
			1&\mbox{ otherwise.}
		\end{cases}
	\]
	The parameter $\eps>0$ and the function $\rho$ are to be determined. We assume that $\rho(s,0)=0$ and $\rho(s,t)=1$ when $t\geq \eps$. We will choose $\eps=o(h^\frac{1}{2})$. 
	
	Since the metric induced by the change of variable is the Euclidean metric modulo $\mathscr{O}(\eps)$, we get
	\begin{multline*}
		h^2\int_{\mathsf{supp \nabla\chi}} |w_h|^2e^{-2\phi/h}|\nabla\chi|^2 \dd x\\	
		\leq (1+\mathscr{O}(\eps)) h^{2}\int_{\Gamma}\int_{0}^{\eps} |\tilde w_h|^2e^{-2\tilde \phi(s,t)/h}(\left|\partial_{t}\rho\right|^2+\left|\partial_{s}\rho\right|^2) \dd s \dd t\,,
	\end{multline*}
	where $\tilde w_h = w_h\circ \eta$ and $\tilde \phi = \phi\circ \eta$.
	Thus, by using the Taylor expansion of $\tilde\phi$ at $t=0$, we get uniformly in $s\in\Gamma$,
	\[\tilde \phi(s,t)=t\partial_{t}\tilde\phi(s,0)+\mathscr{O}(t^2)=-t\partial_{\textbf{n}}\phi(s,0)+\mathscr{O}(\eps^2)\,,\]
	and
	\begin{multline*}
		h^2\int_{\mathsf{supp \nabla\chi}} |w_h|^2e^{-2\phi/h}|\nabla\chi|^2 \dd x\\
		\leq (1+\mathscr{O}(\eps+\eps^2/h)) h^{2}\int_{\Gamma} \int_{0}^{\eps}|\tilde w_h|^2e^{2t\partial_{\textbf{n}}\phi(s)/h}(\left|\partial_{t}\rho\right|^2+\left|\partial_{s}\rho\right|^2) \dd s \dd t\,.
	\end{multline*}
	Since $Q_{n,m}\in W^{1,\infty}(\Omega)$, we have $\partial_t\tilde Q_{n,m}\circ \eta\in L^\infty(\Gamma\times (0,\eps))$ and by using the Taylor expansion of $\tilde w$ near $t=0$, we get
	\[\begin{split}
		\tilde w_h(s,t)  
		&= \left(\sum_{j=0}^{k-1}c_jw_{j,h}\right)\circ \eta(s,t)
		= \tilde w_h(s,0)
		+\int_0^t\partial_t\tilde w_h(s,t')\dd t'
		\\&= \tilde w_h(s,0)
		+\mathscr{O}(\eps)\norm{c_h}_{\ell^2}\,,
	\end{split}\]
	where 
	\begin{equation}\label{eq.ch}
	c_h = (h^{-\frac{1}{2}}c_{0},\dots,h^{-\frac{k}{2}}c_{k-1})\,,
	\end{equation} 
	and $\norm{\cdot}_{\ell^2}$ is the canonical Euclidian norm on $\CC^k$.   
	Then,
	\begin{multline}\label{eq:estitubularup}
		h^2\int_{\mathsf{supp \nabla\chi}} |w_h|^2e^{-2\phi/h}|\nabla\chi|^2 \dd x\\
		\leq (1+\mathscr{O}(\eps+\eps^2/h)) h^{2}\int_{\Gamma} |\tilde w_h(s,0)|^2\int_{0}^{\eps}e^{2t\partial_{\textbf{n}}\phi(s)/h}(\left|\partial_{t}\rho\right|^2+\left|\partial_{s}\rho\right|^2) \dd s \dd t\,
		\\+Ch^{2}\varepsilon \norm{c_h}_{\ell^2}^2\int_{\Gamma}\int_{0}^{\eps}e^{2t\partial_{\textbf{n}}\phi(s)/h}(\left|\partial_{t}\rho\right|^2+\left|\partial_{s}\rho\right|^2) \dd s \dd t\,.
		\end{multline}
	\item
	For the right hand side of \eqref{eq:estitubularup} to be small, we choose $\rho$ to minimize $\partial_t \rho$ far from the boundary. The optimization of 
	\[
		\rho \mapsto \int_{0}^{\eps}e^{2t\partial_{\textbf{n}}\phi/h}\left|\partial_{t}\rho\right|^2\dd t\,,
	\]
	gives us the weight $\partial_\n \phi$. More precisely, 
	Lemma \ref{lem.opt} with $\alpha=2\partial_{\textbf{n}}\phi/h>0$ suggests to consider the trial state defined, for $t\leq \eps$, by
	\[\rho(s,t)=\frac{1-e^{-2t\partial_{\textbf{n}}\phi(s)/h}}{1-e^{-2\eps\partial_{\textbf{n}}\phi(s)/h}}\,,\]
	and by $1$ otherwise. By Lemma \ref{lem.opt}, we get
	\[
		\int_{0}^{\eps}e^{2t\partial_{\textbf{n}}\phi/h}\left|\partial_{t}\rho\right|^2\dd t = 		\frac{2\partial_{\textbf{n}}\phi/h}{1-e^{-\eps2\partial_{\textbf{n}}\phi/h}}\,,
	\]
	and
	\[\begin{split}
		\int_{0}^{\eps}e^{2t\partial_{\textbf{n}}\phi/h}\left|\partial_{s}\rho\right|^2\dd t 
		 = |\pa_s \alpha|^2\int_{0}^{\eps}e^{2t\partial_{\textbf{n}}\phi/h}\left|\partial_{\alpha}\rho_{\alpha,\eps}\right|^2\dd t 
		&\leq  Ch^{-2}\left(
		\alpha^{-3} + e^{-\alpha\eps}\eps^2\alpha^{-1}
		\right)
		\\&\leq C\left(
			h + e^{-\eps2\partial_{\textbf{n}}\phi/h}\eps^2h^{-1}
		\right)\,.
	\end{split}\]
	We can choose $\eps=h|\log h|$  
	so that 
		\[
			\int_\Gamma \int_0^\eps  |\tilde w_h(s,0)|^2 e^{2t\partial_{\textbf{n}}\phi(s)/h}\left|\partial_{t}\rho\right|^2\dd s \dd t = (1+o(1))h^{-1}\int_\Gamma 2\pa_\n\phi|\tilde w_h(s,0)|^2\dd s\,,
		\]
	and	
	 \eqref{eq:estitubularup} becomes
	\begin{equation}\label{eq.ubfq}
		h^2\int_{\mathsf{supp \nabla\chi}} |w_h|^2e^{-2\phi/h}|\nabla\chi|^2 \dd x
		\leq (1+o(1))h\left(\int_\Gamma 2\pa_\n\phi|\tilde w_h(s,0)|^2\dd s
		+C\eps\norm{c_h}_{\ell^2}^2\right)\,.
	\end{equation}
	\item

	 Let us consider, for all $h\geq 0$,
	\[
		N_h:c\in\CC^k\mapsto \left(\int_\Gamma \pa_\n\phi\left|
		\sum_{j=0}^{k-1}c_jh^{\frac{1+j}{2}}\tilde w_{j,h}(s,0)
		\right|^2\dd s\right)^{1/2}\,,
	\] 
	where we recall that
	\[w_{j,h}(z)=h^{-\frac{1}{2}}P_{j}\left(\frac{z-z_{\min}}{h^{\frac{1}{2}}}\right)-h^{-\frac{j+1}{2}}Q_{j,m}(z)\,.\]
	The application $\mathbb{C}^k\times[0,1]\ni (c,h)\mapsto N_{h}(c)$ is well defined and continuous (since the degree of $P_{j}$ is $j$). Note, in particular, that
	\[
		N_0(c)=\left(\int_\Gamma \pa_\n\phi\left|
		\sum_{j=0}^{k-1}c_j[(z-z_{\min})^j-Q_{j,m}(z)]
		\right|^2\dd s\right)^{1/2}\,.
	\]

	Notice that  
	\begin{equation}\label{eq.Nhch}
	N_{h}(c_{h})^2=\int_\Gamma \pa_\n\phi|\tilde w_h(s,0)|^2\dd s=N^2_{\mathcal{H}}( w_{h})\,, 
	\end{equation}
	where $c_{h}$ is defined in \eqref{eq.ch}. Since $N_{\mathcal{H}}$ is a norm, and recalling Remark \ref{rem.free}, we see that the application $N_{h}$ is a norm when $h\in(0,h_{0}]$. $N_{0}$ is also a norm (as we can see by using the Hardy norm and $Q_{j,m}\in\mathscr{H}_{k}^2(\Omega)$).
	
Let us define
	\[
		C_0 = \min_{
			\begin{array}{c}
				h\in[0,h_{0}]\\ \|c\|_{\ell^2=1}
			\end{array}}
			N_h(c)>0\,.
	\]
so that, for all $h\in[0,h_{0}]$, and all $c\in \CC^k$,
	\begin{equation}\label{eq.C0}
	C_{0}\norm{c}_{\ell^2}\leq N_h(c)\,.
	\end{equation}
	
\end{enumerate}
Using \eqref{eq.ubfq}, \eqref{eq.Nhch}, and replacing $c$ by $c_{h}$ in \eqref{eq.C0}, we conclude that
	\[\begin{split}
		h^2\int_{\mathsf{supp \nabla\chi}} |w_h|^2e^{-2\phi/h}|\nabla\chi|^2 \dd x
		&
		\leq 2(1+o(1))h\Nh(w_h)^2
		\,.
	\end{split}\]
	Let us now estimate $\Nh(w_h)$. From the triangle inequality, we get
	\[\Nh(w_h)\leq |c_{k-1}|\Nh(w_{k-1,h})+\sum_{j=0}^{k-2}|c_{j}|\Nh(w_{j,h})\,.\]
	Then, from degree considerations and the triangle inequality, we get, for $1\leq j\leq k-2$,
	\[\Nh(w_{j,h})=\mathscr{O}\left(h^{\frac{1-k}{2}}\right)\,,\]
	and
	\[\Nh(w_{k-1,h})=(1+o(1))h^{-\frac{k}{2}}\Nh\left((z-z_{\min})^{k-1}-Q_{k-1,m}\right)\,.\]
	Then, 
	\[\Nh(w_{h})^2\leq |c_{k-1}|^2h^{-k}\Nh\left((z-z_{\min})^{k-1}-Q_{k-1,m}\right)^2+o(h^{-k})\|c\|^2_{\ell^2}\,.\]

	This ends the proof.
\end{proof}

\subsection{Proof of Proposition \ref{prop:upper-bound}}

Let us define
$
	\widetilde V_{h,k,\rm sup} = \{\chi_h w_h,\, w_h\in V_{h,k,\rm sup}\},
$
where $V_{h,k,\rm sup}$ is defined in \eqref{eq.spaceupbd} and $\chi_h$ in Lemma \ref{lem:estiannularUpper}. 
By Lemmas \ref{lem:L2esti1} and \ref{lem:estiannularUpper}, we get
	\[
	\frac{4\int_{\Omega}h^2e^{-2\phi/h} |\partial_{\overline{z}}(w_h\chi_h)|^2\dd x}{\int_{\Omega}|w_h\chi_h|^2e^{-2(\phi-\phi_{\min})/h}\dd x}
	 \leq 2h^{1-k}\frac{|c_{k-1}|^2\Nh\left((z-z_{\min})^{k-1}-Q_{k-1,m}\right)^2}{\sum_{j=0}^{k-1}|c_j|^2\Nb(P_j)^2}+o(h^{1-k})\,,
	\]
for all $w_h = \sum_{j=0}^{k-1}c_j w_{j,h}\in V_{h,k,\rm sup}$ with $c\in\mathbb{C}^k\setminus\{0\}$. From the min-max principle\footnote{By Remark \ref{rem.free},
$
	\dim \widetilde V_{h,k,\rm sup} = k
$ for $h$ small enough.}, it follows
\[\lambda_k(h)\leq 2h^{1-k}\Nh\left((z-z_{\min})^{k-1}-Q_{k-1,m}\right)^2\sup_{c\in\mathbb{C}^k\setminus\{0\}}\frac{|c_{k-1}|^2}{\sum_{j=0}^{k-1}|c_j|^2\Nb(P_j)^2}e^{2\phi_{\min}/h}+o(h^{1-k})\,.\]
Since 
\[\sup_{c\in\mathbb{C}^k\setminus\{0\}}\frac{|c_{k-1}|^2}{\sum_{j=0}^{k-1}|c_j|^2\Nb(P_j)^2}=\Nb(P_{k-1})^{-2}\,,\]
we deduce
\[
	\limsup_{h\to 0}	h^{k-1}e^{-2\phi_{\min}/h}\lambda_{k}(h)\leq  2\left(\frac{\Nh\left(
			(z-z_{\rm min})^{k-1}
			-Q_{k-1,m}
		\right)}{\distb \left(z^{k-1},\mathcal{P}_{k-2}\right)}\right)^2\,.
\]
Taking the limit $m\to+\infty$, it follows
\[
	\limsup_{h\to 0}	h^{k-1}e^{-2\phi_{\min}/h}\lambda_{k}(h)\leq  C_{\sup}(k)\,.
\]

\subsection{Computation of $C_{\sup}(k)$ in the radial case}
Let $k\in \NN^*$. Let us assume that $\Omega$ is the disk of radius $R$ centered at $0$, and that $B$ is radial. In this case $x_{\min}=0$, $\partial_\n\phi$ is constant and $\mathsf{Hess}_{x_{\min}}\phi = 
B(0){\rm Id}/2$.

Thus,
\[\disth((z-z_{\rm min})^{k-1}, \mathscr{H}^2_{k}(\Omega))=\disth(z^{k-1}, \mathscr{H}^2_{k}(\Omega))=\Nh(z^{k-1})^2=2\pi\partial_{\n}\phi R^{2k-1}\,,\]
and we notice that $P_{n}(z)=z^n$ (see Notation \ref{not.Pn}) so that
\[\begin{split}
	\distb\left(z^{k-1},\mathcal{P}_{k-2}\right) &= \Nb(P_{k-1})^2 \\
	&= \int_{\mathbb{R}^2} \left|y\right|^{2(k-1)}e^{-\mathsf{Hess}_{x_{\min}}\phi(y, y)} \dd y=2\pi\int_{0}^{+\infty}\rho^{2k-1}e^{-
B(0)\rho^2/2}\mathrm{d} \rho\\	
	&=\frac{2\pi2^k}{B(0)^k}\int_{0}^{+\infty} \rho^{2k-1}e^{-\rho^2}\mathrm{d}\rho
	=\frac{2\pi2^{k-1} \Gamma(k)}{B(0)^k}=\frac{2\pi2^{k-1} (k-1)!}{B(0)^k}\,,
\end{split}\]
We get
\[C_{\sup}(k)
=\frac{B(0)^k\Phi R^{2k-2}}{2^{k-2}(k-1)!}
\,.\]
 Note that this formula extends the upper bound obtained in \cite{HP17} for constant magnetic fields on the disc.

\section{On the magnetic Cauchy-Riemann operators}\label{sec.magnCR}

In this section, $U$ will denote an open bounded subset of $\mathbb{R}^2$. It will be either $\Omega$ itself, or a smaller open set.

As we already observed (see \eqref{eq:defDP}), the Dirichlet-Pauli operator, considered only as a differential operator, is the square of the magnetic Dirac operator $\sigma\cdot(\p-\mathbf{A})$. It can be written as
\begin{equation}\label{defi:zigzag}
\sigma\cdot (\p-A)= 
\begin{pmatrix}
0&d_{h,A}\\
{d^{\times}_{h,A}}&0
\end{pmatrix}
\end{equation} 
where $d_{h,A}$ and $d^{\times}_{h,A}$ are the magnetic Cauchy-Riemann operators:
\[d_{h,A}=-2ih\partial_{z}-A_{1}+iA_{2}\,,\qquad d^{\times}_{h,A}=-2ih\partial_{\overline{z}}-A_{1}-iA_{2}\,.\]

Let $(d_{h,A},\mathsf{Dom}(d_{h,A}))$ be the operator on $L^2(U, \CC)$ acting as $d_{h,A}$ on $\mathsf{Dom}(d_{h,A}) = H^1_0(U;\CC)$.
\subsection{Properties of $d_{1,0}$ and $d_{1,0}^*$}
In this part, we study the operators $d_{h,A}$ and $d_{h,A}^*$ in the non-magnetic case $B=0$ with $h=1$ in order to get describe their properties in this simplified setting in which $-\Delta = d_{1,0}^*d_{1,0}$. Various aspects of this section can be related to the spectral analysis of the \enquote{zig-zag} operator (see \cite{S95}). 
The next section will be related to the magnetic case that is needed in our study.
\begin{lemma}\label{lem.zigzag_nm}
	Assume that $U$ is of class $\mathscr{C}^2$.
	The following properties hold.
	\begin{enumerate}[\rm (a)]
		\item The operator $(d_{1,0},\mathsf{Dom}(d_{1,0}))$ is closed with closed range.
		\item The domain of $d_{1,0}^*$ is given by
		\[\begin{split}
			\mathsf{Dom}(d_{1,0}^*) 
			&= \{u\in L^2(U;\CC)\,, \partial_{\overline{z}}u\in L^2(U;\CC) \}
			\\&=  \{u\in L^2(U;\CC)\,, \partial_{\overline{z}}u=0 \}+H^1(U;\CC)\,,
		\end{split}\]
		and $d^*_{1,0}$ acts as $d_{1,0}^\times$. In particular,
		\[
		\ker(d_{1,0}^*)= \{u\in L^2(U;\CC)\,, \partial_{\overline{z}}u=0 \}\,.
	\]
		\item We have
		\[\begin{split}
			\ker(d_{1,0}^*)^\perp\cap \mathsf{Dom}(d_{1,0}^*)  
			= \{d_{1,0}w\,, w\in H^1_0(U;\CC)\cap H^2(U;\CC)\}\subset H^1(U;\CC)\,,
		\end{split}\]	
		and there exists $C>0$ such that, for all $v\in \ker(d_{1,0}^*)^\perp\cap \mathsf{Dom}(d_{1,0}^*)$,
		\[
			\norm{v}_{H^1(U)}\leq C\norm{d_{1,0}^* v}_{L^2(U)}\,.
		\]
	\end{enumerate}
\end{lemma}
\begin{proof}

		Let $u\in \mathsf{Dom}(d_{1,0}) = H^1_0(U;\CC)$. One easily checks that
		\[\begin{split}
			\norm{d_{1,0}u}_{L^2(U)}^2 = \norm{\nabla u}_{L^2(U)}^2\,.
		\end{split}\]
		Hence, the Poincar\'e inequality ensures that $(d_{1,0},\mathsf{Dom}(d_{1,0}))$ is a closed operator with closed range.
Then, by definition of the domain of the adjoint,
		\[
			\mathsf{Dom}(d_{1,0}^*) \subset \{u\in L^2(U;\CC)\,, \partial_{\overline{z}}u\in L^2(U;\CC) \}\,.
		\]		
		Conversely, if $v\in\{u\in L^2(U;\CC)\,, \partial_{\overline{z}}u\in L^2(U;\CC) \}$, we have, for all $w\in\mathscr{C}^\infty_{0}(U)$,
		\[
			\braket{v, -2i\partial_zw}_{L^2(U)}=\braket{-2i\partial_{\overline z}v,w}_{L^2(U)}\,.
		\]
		By density, this equality can be extended to $w\in H^1_0(U;\CC)$. This shows, by definition, that $v\in \mathsf{Dom}(d_{1,0}^*)$ and $d_{1,0}^*v=-2i\partial_{\overline{z}}v$.
		
		Moreover, we have
		\[\begin{split}
			&\ker(d_{1,0}^*)^\perp\cap \mathsf{Dom}(d_{1,0}^*) 
			= {\rm ran}(d_{1,0})\cap \mathsf{Dom}(d_{1,0}^*) 
			\\&= \{d_{1,0}w\,, w\in H^1_0(U;\CC) \mbox{ and }  -2i\partial_{\overline{z}}(d_{1,0}w) = -\Delta w\in L^2(U;\CC)\}
			\\&= \{d_{1,0}w\,, w\in H^1_0(U;\CC)\cap H^2(U;\CC)\}\subset H^1(U;\CC)\,,
		\end{split}\]
		where the last equality follows from the elliptic regularity of the Laplacian. In particular, we get, for all $w\in H^1_0(U;\CC)\cap H^2(U;\CC)$,
		\[
			\norm{w}_{H^2(U)}\leq C\norm{\Delta w}_{L^2(U)}\,.
		\]
		Now, take $v\in \ker(d_{1,0}^*)^\perp\cap \mathsf{Dom}(d_{1,0}^*)$. We can write $v=d_{1,0}w$ with $w\in H^2(U; \mathbb{C})\cap H^1_{0}(U;\mathbb{C})$. We have $d^*_{1,0}v=-\Delta w$ so that 
		\[\norm{v}_{H^1(U)}\leq C\norm{d_{1,0}^* v}_{L^2(U)}\,.\]
\end{proof}
\subsection{Properties of $d_{h,A}$ and $d_{h,A}^*$}
Let us introduce some notations related with the Riemann mapping theorem.
In the following, we gather some standards properties related with $d_{h,A}$ and $d_{h,A}^*$.  We will use the following lemma.

\begin{lemma}\label{eq.b4}
For all $u\in\mathscr{C}^\infty_{0}(U;\mathbb{C})$, we have
\[\begin{split}
\norm{d_{h,A}u}_{L^2(U)}^2&= \norm{(\p-A)u}_{L^2(U)}^2 + h\int_UB|u|^2\dd x\\
\norm{d^\times_{h,A}u}_{L^2(U)}^2&= \norm{(\p-A)u}_{L^2(U)}^2-h\int_UB|u|^2\dd x
\end{split}\,.
\]
These formulas can be extended to $u\in H^1_{0}(U;\mathbb{C})$.
\end{lemma}
\begin{proof}
It follows from integrations by parts and the fact that $d_{h,A}d^\times_{h,A}=|\p-A|^2-hB$ and $d^\times_{h,A}d_{h,A}=|\p-A|^2+hB$. The extension to $u\in H^1_{0}(U;\mathbb{C})$ follows by density.
\end{proof}
\begin{remark}\label{rem.boundmagn}
	From Lemma \ref{eq.b4}, we deduce\footnote{It may also be found in \cite[Lemma 1.4.1]{FH11}.}, that for all $u\in H^1_0(U;\CC)$,
	\[\|(\p-A)u\|_{L^2(U)}^2\geq \int_{U} hB|u|^2\dd x\,.\]
\end{remark}

\begin{proposition}\label{prop:zigzagtot}
	Assume that $U$ is of class $\mathscr{C}^2$. 
	\begin{enumerate}[\rm (a)]
		\item\label{eq.b1} The operator $(d_{h,A},\mathsf{Dom}(d_{h,A}))$ is closed with closed range.
		\item\label{eq.b2} The adjoint $(d_{h,A}^*,\mathsf{Dom}(d_{h,A}^*))$ acts as $d_{h,A}^\times$ on
		\[\begin{split}
			\mathsf{Dom}(d^*_{h,A})
			&=\{u\in L^2(U) : \partial_{\overline z} u\in L^2(U)\} = \ker(d^*_{h,A}) + H^1(U;\CC)\,
		\end{split}\]
		and
		\[
			\ker(d^*_{h,A}) = \{e^{-\phi/h}v\,, v\in L^2(U), \partial_{\overline{z}}v = 0\}\,.
		\]		
	
		\item\label{eq.b3}  We have $\ker(d_{h,A}^*)^\perp\cap \mathsf{Dom}(d_{h,A}^*)= \{d_{h,A}w\,, w\in H^1_0(U;\CC)\cap H^2(U;\CC)\}$.
	\end{enumerate}
\end{proposition}
\begin{notation}\label{not.5}
	 The notation $d_{h,A,U}$ for $d_{h,A}$ emphasizes the dependence on $U$.
	We denote by $\Pi_{h,A,U}$ (or simply $\Pi_{h,A}$ if there is no ambiguity) the orthogonal projection on $\ker(d^*_{h,A})$.
\end{notation}
\begin{proof}
\begin{enumerate}[\rm (a)]
\item 

By Lemma \ref{eq.b4}, the graph norm of $d_{h,A}$ and the usual $H^1$-norm are equivalent. Thus, the graph of $d_{h,A}$ is a closed subspace of $L^2(U)\times L^2(U)$.
From Lemma \ref{eq.b4} and Remark \ref{rem.boundmagn}, we get, for all $u\in H^1_{0}(U)$,
\[\|d_{h,A}u\|^2_{L^2(U)}\geq h\int_{U}2B|u|^2\dd x\,.\]
With Assumption \eqref{eq.a1} of Theorem \ref{theo.main} and the fact that the operator is closed, the range is closed.
\item 
We have
\[
	\mathsf{Dom}(d^*_{h,A}) = \mathsf{Dom}(d^*_{1,0})\,,
\]
and $d^*_{h,A}$ acts as $d^\times_{h,A}$.
 By Proposition \ref{prop:zeromode} and Lemma \ref{lem.zigzag_nm}, we deduce
		\[
			\ker(d^*_{h,A})  = \{e^{-\phi/h}v\,, v\in L^2(U), \partial_{\overline{z}}v = 0\}\,.
		\]
\item As in the proof of Lemma \ref{lem.zigzag_nm}, we get 
\[\begin{split}
			&\ker(d_{h,A}^*)^\perp\cap \mathsf{Dom}(d_{h,A}^*) 
			= {\rm ran}(d_{h,A})\cap \mathsf{Dom}(d_{h,A}^*) 
			\\&= \{d_{h,A}w\,, w\in H^1_0(U;\CC) \mbox{ and }  d_{h,A}^*d_{h,A}w = (|p-A|^2+hB) w\in L^2(U;\CC)\}
			\\&= \{d_{h,A}w\,, w\in H^1_0(U;\CC) \mbox{ and }  -\Delta w\in L^2(U;\CC)\}
			\\&= \{d_{h,A}w\,, w\in H^1_0(U;\CC)\cap H^2(U;\CC)\}\,.
\end{split}\,.\]
\end{enumerate}

\end{proof}

\begin{definition}
We define the self-adjoint operators $(\mathscr{L}^\pm_{h},\mathsf{Dom}(\mathscr{L}^\pm_{h}))$ as the operators acting as 
		\begin{equation}\label{eq.dd}
			 \mathscr{L}^-_{h} = d_{h,A} d_{h,A}^\times=|\p-A|^2-hB\,, \quad\mathscr{L}^+_{h} = d_{h,A}^\times d_{h,A}=|\p-A|^2+hB\,,
		\end{equation}
		on the respective domains
		\[\begin{split}
			\mathsf{Dom}(\mathscr{L}^-_{h}) 
			&= \{u\in\mathsf{Dom}(d_{h,A}^*)\,, d_{h,A}^*u\in\mathsf{Dom}(d_{h,A})\}\,,
			\\
			\mathsf{Dom}(\mathscr{L}^+_{h}) 
			&= \{u\in\mathsf{Dom}(d_{h,A})\,, d_{h,A}u\in\mathsf{Dom}(d_{h,A}^*)\} 
			\\&= H^1_0(U;\CC)\cap H^2(U;\CC)\,.
		\end{split}\]
\end{definition}

\subsection{Semiclassical elliptic estimates for the magnetic Cauchy-Riemann operator}
\begin{notation}\label{not.RM}
By the Riemann mapping theorem, and since $\partial\Omega$ is assumed to be $\mathscr{C}^2$, it is Dini-continuous (see \cite[Theorem 2.1, and Section 3.3]{P92}) and we can consider a biholomorphic function $F$ between  $D(0,1)$ and $\Omega$ such that $F(\partial D(0,1))=\partial\Omega$. We write $x=F(y)$. We notice that
\[\partial_{y_{1}}+i\partial_{y_{2}}=\overline{F'(y)}(\partial_{x_{1}}+i\partial_{x_{2}})\,,\]
and
\[\dd x=|F'(y)|^2\dd y\,.\]
By \cite[Theorem 3.5]{P92}, this biholomorphism can be continuously extended to $\overline{D(0,1)}$, and there exist $c_{1}, c_{2}>0$ such that, for all $y\in\overline{D(0,1)}$,
\[c_{1}\leq |F'(y)|\leq c_{2}\,.\]
For $\delta\in(0,1)$, we also let
\[
\Omega_{\delta}=
F(D(0,1-\delta))
\,.\]
Note that $\Omega_\delta$ is actually an analytic manifold.
\end{notation}

The following theorem is a crucial ingredient in the proof of the lower bound of $\lambda_{k}(h)$. It is intimately related to the spectral supersymmetry of Dirac operators \cite[Theorem 5.5 and Corollary 5.6]{T92}.
\begin{theorem}\label{prop.elliptic}
There exist $\delta_{0}, h_0>0$ and $c>0$ such that, for all $\delta\in[0,\delta_{0})$, for all $h\in(0,h_0)$, and for all $u\in \mathsf{Dom}(d_{h,A, \Omega_{\delta}}^*)\cap \ker(d_{h,A,\Omega_{\delta}}^*)^\perp$
\[\| d_{h,A,\Omega_{\delta}}^* u\|_{L^2(\Omega_{\delta})}\geq \sqrt{2hB_0}\|u\|_{L^2(\Omega_\delta)}\,,\]
\[
\| d_{h,A,\Omega_{\delta}}^* u\|_{L^2(\Omega_{\delta})}\geq c
h^2\left(\norm{\nabla u}_{L^2(\Omega_\delta)}+\norm{u}_{L^2(\partial \Omega_\delta)}\right)\,,
\]
where we used Notation \ref{not.RM}.
\end{theorem}
Theorem \ref{prop.elliptic} follows from the following two lemmas.
\begin{lemma}\label{prop.orth.proj}
	For all $u\in\mathsf{Dom}( d_{h,A,\Omega_{\delta}}^*)\cap \ker( d_{h,A,\Omega_{\delta}}^*)^\perp$, we have
	\[\| d_{h,A,\Omega_{\delta}}^* u\|_{L^2(U)}\geq \sqrt{2hB_0}\|u\|_{L^2(U)}\,.\]
\end{lemma}
\begin{proof}
	
	For notational simplicity, we let $U=\Omega_{\delta}$ and we write $d_{h,A}$ for $d_{h,A,U}$.
	Let $u\in\mathsf{Dom}(d^*_{h,A})\cap \ker(d_{h,A}^*)^\perp$. By Proposition \ref{prop:zigzagtot}, there exists $w\in H^1_0(U;\CC)\cap H^2(U;\CC)$ such that $u = d_{h,A}w$ and $d_{h,A}^* u = \mathscr{L}^+_{h}w$.
	The spectrum of $\mathscr{L}^+_{h}$ is a subset of $[2hB_0,+\infty)$ (see Remark \ref{rem.boundmagn}).
	Thus, we get
	\[\| \mathscr{L}^+_{h}w\|_{L^2(U)}\geq 2hB_{0}\|w\|_{L^2(U)}\,.\]
	By integration by parts and the Cauchy-Schwarz inequality, we have
	\[\begin{split}
	2hB_0\norm{d_{h,A}w}_{L^2(U)}^2
	&
	\leq 2hB_0\braket{w, \mathscr{L}^+_{h}w}_{L^2(U)}
	\leq 2hB_0\norm{w}_{L^2(U)}\norm{\mathscr{L}^+_{h}w}_{L^2(U)}
	\\&
	\leq\norm{\mathscr{L}^+_{h}w}_{L^2(U)}^2\,.
	\end{split}\]
	This ensures that 
	\[
	\sqrt{2hB_0}\norm{d_{h,A}w}_{L^2(U)}\leq\norm{\mathscr{L}^+_{h}w}_{L^2(U)} = \norm{{d^*_{h,A}} \left(d_{h,A}w\right)}_{L^2(U)}
	\]
	and the conclusion follows.
\end{proof}
\begin{lemma}\label{prop.orth.proj1}
	There exist $\delta_{0},h_0>0$ and $c>0$ such that, for all $\delta\in[0,\delta_{0})$, for all $h\in(0,h_0)$, and for all $u\in \mathsf{Dom}(d_{h,A, \Omega_{\delta}}^*)\cap \ker(d_{h,A,\Omega_{\delta}}^*)^\perp$, we have
	\[
	\|{d^*_{h,A,\Omega_{\delta}}}u\|_{L^2(\Omega_{\delta})} \geq 
	ch^2\norm{\nabla u}_{L^2(\Omega_\delta)}+ch^{2}\norm{u}_{L^2(\partial \Omega_\delta)}\,.
	\]
\end{lemma}
\begin{proof}
For notational simplicity, we let $U=\Omega_{\delta}$ and we write $d_{h,A}$ for $d_{h,A,U}$.

With the same notations as in the proof of Lemma \ref{prop.orth.proj}  ($u = d_{h,A}w$),  we have
\[d^*_{h,A}u=d^*_{h,A}d_{h,A}w=\mathscr{L}^+_{h}w\,,\qquad w\in H^1_0(U)\cap H^2(U)\,. \]
	\begin{enumerate}[\rm i.]
		\item From Lemma \ref{eq.b4},
		\[\begin{split}
		\|d_{h,A}w\|^2_{L^2(U)}
		&= \norm{(\p-A)w}_{L^2(U)}^2 + h\int_U B|w|^2\dd x
		\\&= \braket{d^*_{h,A}u,w}_{L^2(U)} \leq \norm{d^*_{h,A}u}_{L^2(U)}\norm{w}_{L^2(U)}\,.
		\end{split}\]
		Using Assumption \eqref{eq.a1}, we get
		\[
			B_0h\|w\|^2_{L^2(U)}\leq h\int_U B|w|^2\dd x\leq  \norm{d^*_{h,A}u}_{L^2(U)}\norm{w}_{L^2(U)}\,,
		\]
		and
		\begin{equation}\label{eq.777}
			h\|w\|_{L^2(U)}\leq B_0^{-1}\norm{d^*_{h,A}u}_{L^2(U)}\,.
		\end{equation}
		Since
		\[
			h\int_U B|w|^2\dd x\leq\norm{(\p-A)w}_{L^2(U)}^2\,,
		\]
		we deduce that
		\[
			B_0^{1/2}h^{1/2}\norm{w}_{L^2(U)}\leq\norm{(\p-A)w}_{L^2(U)}\,,
		\]
		and
		\begin{equation}\label{eq.778}
			 \norm{(\p-A)w}_{L^2(U)}^2\leq \norm{d^*_{h,A}u}_{L^2(U)}\norm{w}_{L^2(U)}\leq  \norm{d^*_{h,A}u}_{L^2(U)}B_0^{-1/2}h^{-1/2}\norm{(\p-A)w}_{L^2(U)}\,,
		\end{equation}
		so that using \eqref{eq.777} and \eqref{eq.778},
		there exists $C>0$ such that
		\begin{equation*}
		h^{\frac{1}{2}}\norm{(\p-A)w}_{L^2(U)} + h\norm{w}_{L^2(U)}\leq  C\norm{d^*_{h,A}u}_{L^2(U)}\,.
		\end{equation*}
		Since $A$ is bounded,
		\begin{equation*}
		h^{3/2}\norm{\nabla w}_{L^2(U)}\leq C\norm{d^*_{h,A}u}_{L^2(U)} + Ch^{1/2}\norm{w}_{L^2(U)}\leq Ch^{-1/2}\norm{d^*_{h,A}u}_{L^2(U)}\,.
		\end{equation*}
		Thus,
		\begin{equation}\label{esti.normh1CR}
		h^{2}\norm{\nabla w}_{L^2(U)} + h\norm{w}_{L^2(U)}\leq  C\norm{d^*_{h,A}u}_{L^2(U)}\,.
		\end{equation}
		\item Let us now deal with the derivatives of order two. From the explicit expression of $\mathscr{L}^+_h w$, we get
		\[\begin{split}
		-h^2\Delta w = d^*_{h,A}u -2ihA\cdot\nabla w-|A|^2w+hBw
		\end{split}\,.\]
		Taking the $L^2$-norm and using \eqref{esti.normh1CR}, we get
		\[\begin{split}
		h^2\norm{\Delta w}_{L^2(U)}
		&\leq \norm{d^*_{h,A}u}_{L^2(U)}+\norm{-2ihA\cdot\nabla w}_{L^2(U)}+\norm{|A|^2w}_{L^2(U)}+\norm{hBw}_{L^2(U)}
		\\&\leq C(1+h^{-1})\norm{d^*_{h,A}u}_{L^2(U)}\,.
		\end{split}\]
		Using a standard ellipticity result for the Dirichlet Laplacian, we find
		\begin{equation}\label{esti.normh2CR}
		h^{3}\norm{w}_{H^2(U)}+h^{2}\norm{\nabla w}_{L^2(U)} + h\norm{w}_{L^2(U)}\leq  C\norm{d^*_{h,A}u}_{L^2(U)}\,.
		\end{equation}
		The uniformity of the constant with respect to $\delta\in(0,\delta_0)$ can be checked in the classical proof of elliptic regularity.
		Alternatively, using Riemann mapping theorem, we send $\Omega$ on the unit disk. Then, we perform a change of scale for each $\delta$ to  send $D(0,1-\delta)$ onto $D(0,1)$ and use a standard ellipticity result on $D(0,1)$. Here, $\delta$ appears as a regular parameter in the coefficients of the elliptic operator.
		Note that $d_{h,A}=L_{1}-i L_{2}$ where $L_{j}=-ih\partial_{j}-A_{j}$.
		Using \eqref{esti.normh2CR}, we deduce that
		\begin{equation*}
		\begin{split}
		\norm{\nabla d_{h,A} w}_{L^2(U)}
		&\leq Ch\norm{w}_{H^2(U)} + C\norm{w}_{L^2(U)} + C\norm{\nabla w}_{L^2(U)}
		\leq Ch^{-2}\norm{d^*_{h,A}u}_{L^2(U)}\,,
		\end{split}
		\end{equation*}
		and, since $u = d_{h,A}w$,
		\begin{equation}\label{eq.nabla}
		h^2\norm{\nabla u}_{L^2(U)}\leq C\norm{d^*_{h,A}u}_{L^2(U)}\,.
		\end{equation}
		\item
		A classical trace result combined with \eqref{eq.nabla} and Lemma \ref{prop.orth.proj} gives 
		\[
	 \norm{u}_{L^2(\partial U)}\leq C\norm{u}_{H^1(U)}\leq Ch^{-2}\|{d^*_{h,A}}u\|_{L^2(U)} \,,
		\]
		where it can again be checked using the same techniques that $C$ does not depend on $\delta\in(0,\delta_0)$.
	\end{enumerate}
\end{proof}

\section{Lower bounds}\label{sec.4}
The aim of this section is to establish the following proposition.
\begin{proposition}\label{prop.lowbd}
Assume that $\Omega$ is $\mathscr{C}^2$ and satisfies {\rm Assumption \ref{asum:setom}}.
There exists a constant $\theta_0\in(0,1]$ such that for all $k\in\mathbb{N}^*$, 
\[\liminf_{h\to 0} e^{-2\phi_{\min}/h}h^{k-1}\lambda_{k}(h)\geq C_{\rm sup}(k)\theta_0=C_{\rm inf}(k)\,.\]
If $\Omega = D(0,1)$ and $B$ is radial, we have
\[\liminf_{h\to 0} e^{-2\phi_{\min}/h}h^{k-1}\lambda_{k}(h)\geq \frac{4\Phi}{(k-1)!}\det(\mathsf{Hess}_{x_{\min}}\phi)^{\frac{k}{2}}\,.\]
\end{proposition}

\subsection{Inside approximation by the zero-modes}

 Let $k\in\mathbb{N}^*$. Let us consider an orthonormal family $(v_{j,h})_{1\leq j\leq k}$ (for the scalar product of $L^2(e^{-2\phi/h}\dd x)$) associated with the eigenvalues $(\lambda_{j}(h))_{1\leq j\leq k}$. We define 
\[\mathscr{E}_{h}=\underset{1\leq j\leq k}{\mathrm{span}}\, v_{j,h}\,.\]
In this section, we will see that the general upper bound proved in the last section implies that all $v_{h}\in\mathscr{E}_{h}$ wants to be holomorphic inside $\Omega$.
\subsubsection{Concentration of the groundstate}
\begin{lemma}\label{lem.vhL2}
There exist $C, h_{0}>0$ such that for all $v_{h}\in\mathscr{E}_{h}$ and $h\in(0,h_{0})$, we have,
\[\|v_{h}\|^2_{L^2(\Omega)}\leq C h^{-(1+k)} e^{2\phi_{\min}/h}\int_{\Omega}e^{-2\phi/h}|v_{h}|^2\dd x\,.\]
\end{lemma}
This result will be used in the proof of Lemma \ref{lem.normL2} to compute the weighted $L^2$ norm of $v_h$ on $\Omega$  in term of its weighted $L^2$ norm on a shrinking neighborhood of $x_{\rm min}$.

\begin{proof}
 We have $\lambda_{k}(h)=h^{-k+1}\mathscr{O}(e^{2\phi_{\min}/h})$  (see Proposition~\ref{prop:upper-bound}). By using the orthogonality of the $v_{j,h}$, one gets
\begin{equation}\label{eq.l1bound}
\int_{\Omega}e^{-2\phi/h}|2h\partial_{\overline{z}}v_{h}|^2\dd x\leq \lambda_{k}(h)\int_{\Omega}e^{-2\phi/h}|v_{h}|^2\dd x\leq Ch^{-k+1}e^{2\phi_{\min}/h}\int_{\Omega}e^{-2\phi/h}|v_{h}|^2\dd x\,.
\end{equation}
Now, we use $\phi\leq 0$ to get
\[\int_{\Omega}|2\partial_{\overline{z}}v_{h}|^2\dd x\leq C h^{-(1+k)} e^{2\phi_{\min}/h}\int_{\Omega}e^{-2\phi/h}|v_{h}|^2\dd x\,.\]
Since $v_{h}$ satisfies the Dirichlet boundary condition and by integration by parts, we find
\[\int_{\Omega}|\nabla v_{h}|^2\dd x\leq C h^{-(1+k)} e^{2\phi_{\min}/h}\int_{\Omega}e^{-2\phi/h}|v_{h}|^2\dd x\,.\]
It remains to use the Poincar\'e inequality.
\end{proof}
We can now prove a concentration lemma.

\begin{lemma}\label{lem.normL2} 
Let $\alpha\in(0,1/2)$.
We have 
\[\lim_{h\to 0}\,\sup_{v_h\in\mathscr{E}_{h}\setminus\{0\}}\left|\frac{\int_{D(x_{\rm min}, \ h^{\alpha})}e^{-2\phi/h}|v_{h}(x)|^2\dd x}{\int_{\Omega}e^{-2\phi/h}|v_{h}(x)|^2\dd x}-1\right|=0\,,\]
and
\[
\lim_{h\to 0}\,\sup_{\delta\in(0,\delta_0]}\,\sup_{v_h\in\mathscr{E}_{h}\setminus\{0\}}\left|\frac{\int_{\Omega_{\delta}}e^{-2\phi/h}|v_{h}(x)|^2\dd x}{\int_{\Omega}e^{-2\phi/h}|v_{h}(x)|^2\dd x}-1\right|=0\,,
\]
where $\delta_0$ is defined in Proposition \ref{prop:zigzagtot}.
\end{lemma}

\begin{proof}
Let us remark that the second limit is a consequence of the first one.
We have
\[
\frac{\int_{D(x_{\rm min}, h^{\alpha})}e^{-2\phi/h}|v_{h}(x)|^2\dd x}{\int_{\Omega}e^{-2\phi/h}|v_{h}(x)|^2\dd x}=1-\frac{\int_{\Omega\setminus D(x_{\rm min},\ h^{\alpha})}e^{-2\phi/h}|v_{h}(x)|^2\dd x}{\int_{\Omega}e^{-2\phi/h}|v_{h}(x)|^2\dd x}\,.
\]
By \eqref{eq.tayext} and Lemma \ref{lem.vhL2}, we deduce that
\[\begin{split}
	&\int_{\Omega\setminus D(x_{\rm min},\ h^{\alpha})}e^{-2\phi/h}|v_{h}(x)|^2\dd x
	\leq 
	e^{-2\phi_{\rm min}/h-\lambda_{\rm min}h^{2\alpha-1}(1+\mathscr{O}(h^\alpha))}\int_{\Omega}|v_{h}(x)|^2\dd x
	\\&\leq
	C h^{-(1+k)}e^{-\lambda_{\rm min}h^{2\alpha-1}(1+\mathscr{O}(h^\alpha))} \int_{\Omega}e^{-2\phi/h}|v_{h}|^2\dd x = \mathscr{O}(h^\infty)\int_{\Omega}e^{-2\phi/h}|v_{h}|^2\dd x\,,
\end{split}\]
and the conclusion follows.
\end{proof}

\subsubsection{Interior approximation}
Now that we know that $v_{h}$ is localized inside $\Omega$, let us explain why it is close to be a holomorphic function.

\begin{notation}
Let us denote by $\widetilde\Pi_{h,\delta}$ the orthogonal projection on the kernel of $-i\partial_{\overline{z}}$ (i.e. the Segal–Bargmann functions on $\Omega_{\delta}$ which is defined in Theorem \ref{prop.elliptic}) for the $L^2$-scalar product $\langle\cdot,e^{-2\phi/h}\cdot \rangle_{L^2(\Omega_{\delta})}$.
\end{notation}
We notice that, if $u=e^{-\phi/h}v$, we have
\[\Pi_{h,A, \Omega_{\delta}}u=e^{-\phi/h}\widetilde{\Pi}_{h,\delta}v\,,\]
where $\Pi_{h,A, \Omega_{\delta}}$ was defined in Notation \ref{not.5}.
\begin{proposition}\label{prop.approxH1norm}
There exist $C, h_{0}>0$ such that for all $\delta\in(0,\delta_{0}]$ and $h\in(0,h_{0})$, we have for all $v_{h}\in\mathscr{E}_{h}$,
\begin{enumerate}[\rm (a)]
	\item \label{prop.ptlowBound1}$\|e^{-\phi/h} (\mathrm{Id}-\tilde\Pi_{h,\delta})v_{h}\|_{L^2(\Omega_{\delta})}\leq Ch^{-\frac{1}{2}}\sqrt{\lambda_{k}(h)}\|e^{-\phi/h}v_{h}\|_{L^2(\Omega_{\delta})}$,
	\item \label{prop.ptlowBound2}$\|e^{-\phi/h}(\mathrm{Id}-\widetilde\Pi_{h,\delta})v_h\|_{L^2(\partial\Omega_{\delta})}\leq C h^{-2}\sqrt{\lambda_{k}(h)} \|e^{-\phi/h}v_h\|_{L^2(\Omega_{\delta})}$,
	\item \label{prop.ptlowBound3}$\dim\widetilde\Pi_{h,\delta}\mathscr{E}_{h} = k$.
\end{enumerate}
Here, $\delta_0$ is defined in Theorem \ref{prop.elliptic}.
\end{proposition}
\begin{proof}
For all $v_{h}\in\mathscr{E}_{h}$, we have
\[\begin{split}
4\|e^{-\phi/h} h\partial_{\overline{z}}v_{h}\|^2_{L^2(\Omega_{\delta})}&\leq 4\|e^{-\phi/h} h\partial_{\overline{z}}v_{h}\|^2_{L^2(\Omega)}\\
&\leq \lambda_{k}(h)\|e^{-\phi/h}v_{h}\|^2_{L^2(\Omega)}\leq (1+o(1))\lambda_{k}(h)\|e^{-\phi/h}v_{h}\|^2_{L^2(\Omega_{\delta})}\,,
\end{split}\]
where we used Lemma \ref{lem.normL2} to get the last inequality. With $u_{h}=e^{-\phi/h}v_{h}$, we have
\[4\|e^{-\phi/h} h\partial_{\overline{z}}v_{h}\|^2_{L^2(\Omega_{\delta})}=4\|e^{-\phi/h} h\partial_{\overline{z}}(\mathrm{Id}-\tilde\Pi_{h,\delta})v_{h}\|^2_{L^2(\Omega_{\delta})}=\|d^*_{h,A,\Omega_{\delta}}(\mathrm{Id}-\Pi_{h,A,\Omega_{\delta}})u_{h}\|^2_{L^2(\Omega_{\delta})}\,.\]
Applying Theorem \ref{prop.elliptic}, we get \eqref{prop.ptlowBound1} and \eqref{prop.ptlowBound2}.

Let $v_{h}\in\mathscr{E}_{h}$ be such that $\widetilde\Pi_{h,\delta} v_h= 0$. Recalling Proposition \ref{prop:upper-bound}, we have
\[\begin{split}
	&\norm{e^{-\phi/h}v_h}_{L^2(\Omega_\delta)}
	\leq \norm{e^{-\phi/h}(\mathrm{Id}-\widetilde\Pi_{h,\delta})v_h}_{L^2(\Omega_{\delta})}+\norm{e^{-\phi/h}\widetilde\Pi_{h,\delta}v_h}_{L^2(\Omega_{\delta})}
	\\
	&
	\leq C h^{-k/2} e^{\phi_{min}/h}\|e^{-\phi/h}v_h\|_{L^2(\Omega_{\delta})}\,,
\end{split}\]
so that
\[
	\norm{e^{-\phi/h}v_h}_{L^2(\Omega_\delta)}(1-C h^{-k/2} e^{\phi_{min}/h})\leq 0\,,
\]
and $v_h = 0$ on $\Omega_\delta$ so that $\widetilde\Pi_{h,\delta}$ is injective on $\mathscr{E}_{h}$ and \eqref{prop.ptlowBound3} follows.
\end{proof}

\subsection{A reduction to a holomorphic subspace}
In the following, we assume that $\delta\in(0,\delta_0)$ and $h\in(0,h_0)$.

\begin{notation}\label{not.szego}
We will use the so-called Szeg\"o projection
\[\Pi_+ : L^2(D(0,1))\ni\left(\sum_{\n\in\ZZ}a_n(r)e^{i n s}\right)\mapsto \left(\sum_{\n\in\NN}a_n(r)e^{i n s}\right)\in L^2(D(0,1))\,.\]
Note that the Szeg\"o projection preserves the $L^2$ holomorphic functions.
\end{notation}

\begin{notation}
We let
\[
 E := \min_{\pa D(0,1)}|F'(y)|\pa_{\n}\phi(F(y))\geq c_1 \min_{\Gamma}\left(\nabla\phi\cdot\mathbf{n} \right)\,,
\]
where $F$, $c_1$ are defined in Notation \ref{not.RM}. 
\end{notation}

\begin{lemma}\label{lem:lowbound_annular_energy}
	Assume that $\delta/h\to+\infty$ and $\delta\to0$. Then, for all $v_{h}\in\mathscr{E}_{h}$,
	\[
		2h E  \norm{\Pi_+ \left(v_h\circ F\right)}_{L^2(\pa D(0,1-\delta))}^2(1+o(1))
		\leq 4h^2\int_{\Omega}e^{-2\phi/h}|\partial_{\overline{z}}v_h|^2\dd x\,.
	\]
\end{lemma}

\begin{proof}
\begin{enumerate}[\rm i.]
\item
For all $v\in H^1_{0}(\Omega)$, we let $\check v=v\circ F\in H^1_{0}(D(0,1))$ and $\check\phi=\phi\circ F$. We get that
\[\begin{split}
	&4h^2\int_\Omega e^{-2\phi/h}|\partial_{\overline z}v|^2\dd x
	=
	4h^2\int_{D(0,1)} e^{-2\check\phi/h}|\partial_{\overline y}\check v|^2\dd y\,.
\end{split}\]
\item
In the polar coordinates, the Cauchy-Riemann operator is
\[-\frac{i}{2}\left(\partial_{1}+i\partial_{2}\right)=\frac{-i\gamma'}{2}\left(\frac{\partial_{s}}{1-t}+i\partial_{t}\right)\,.\]
We write $\tilde \psi (s,t)=\check\psi(\eta(s,t))$ for any function $\check\psi$ defined on $D(0,1)$. We have, for all $\check v\in H^1_{0}(D(0,1))$,
\begin{align}\label{eq.remove.interior}
\nonumber
4h^2\int_{D(0,1)}e^{-2\check\phi/h}|\partial_{\overline{z}}\check v|^2\dd x
&\geq 4h^2\int_{D(0,1)\setminus D(0,1-\delta)}e^{-2\check \phi/h}|\partial_{\overline{z}}\check v|^2\dd x\\
&=h^2\int_{0}^{2\pi}\int_{0}^\delta (1-t)^{-1}\left|((1-t)\partial_{t}-i\partial_{s})\tilde v\right|^2e^{-2\tilde\phi/h} \dd s\dd t\,.
\end{align}
\item
Let us notice that 
\begin{equation}\label{eq.normal}
	\pa_{\check \n}\check \phi(y) = |F'(y)|\pa_{\n}\phi(F(y))\,,
\end{equation}
where $x\in\partial \Omega$, $\n(x)$ is the outward pointing unit normal to $\Omega$ at $x$, and $\check \n(y)$ is  outward pointing unit normal to $D(0,1)$ at $y$.

By using the Taylor expansion of $\tilde\phi$, there exists $C>0$ such that, for all $(s,t)\in\Gamma\times(0,\delta)$,
\[\begin{split}
	-\tilde\phi(s,t)
	=-t\partial_{t}\tilde\phi(s,0)+\mathscr{O}(t^2)
	\geq (1-C\delta) Et \,,
\end{split}\]
where $0< E = \min_{\partial D(0,1)}\pa_{\check \n}\check \phi$. 
We have
\[
4h^2\int_{D(0,1)}e^{-2\check\phi/h}|\partial_{\overline{z}}\check v|^2\dd x
\geq
h^2\int_{0}^{2\pi}\int_{0}^\delta (1-t)^{-1}\left|((1-t)\partial_{t}-i\partial_{s})\tilde v\right|^2e^{2(1-C\delta)E t/h} \dd s\dd t\,.\]
Consider the new variable $\tau=-\ln|1-t|$, we get
\[
	4h^2\int_{D(0,1)}e^{-2\check\phi/h}|\partial_{\overline{z}}\check v|^2\dd x
	\geq 
	h^2\int_{0}^{2\pi}\int_{0}^{-\ln|1-\delta|}\left|(\partial_{\tau}-i\partial_{s}) \mathsf v(s,\tau) \right|^2e^{2 (1-C\delta)E (1-e^{-\tau})/h} \dd s\dd \tau\,,\]
where $\mathsf v(s,\tau)=\tilde v(s,1-e^{-\tau})$. Since $1-e^{-\tau}=\tau+\mathscr{O}(\tau^2)=\tau+\mathscr{O}(\delta\tau)$, there exists $\tilde C>0$ such that
\[e^{2(1-C\delta)E (1-e^{-\tau})/h}\geq e^{2E(1-\tilde C\delta)\tau/h} \,.\]
Let $\tilde E=E(1-\tilde C\delta)$ and $\tilde \delta=-\ln|1-\delta|$ so that
\[
	4h^2\int_{D(0,1)}e^{-2\check\phi/h}|\partial_{\overline{z}}\check v|^2\dd x
	\geq h^2\int_{0}^{2\pi}\int_{0}^{\tilde\delta}\left|(\partial_{\tau}-i\partial_{s})\mathsf v(s,\tau)\right|^2e^{2\tilde E \tau/h} \dd s\dd \tau\,.\]
\item
Using the Fourier series and the Parseval formula, we get
\[
\int_{0}^{2\pi}\int_{0}^{\tilde\delta}\left|(\partial_{\tau}-i\partial_{s})\mathsf v(s,\tau)\right|^2e^{2\tilde E \tau/h} \dd s\dd \tau
=2\pi\sum_{m\in\mathbb{Z}}\int_{0}^{\tilde\delta}\left|(\partial_{\tau}+m)\hat{\mathsf{v}}_{m}(\tau)\right|^2e^{2\tilde E \tau/h}\dd \tau\,,
\]
where 
\[\hat{\mathsf{v}}_{m}(\tau)=\frac{1}{2\pi}\int_{0}^{2\pi} e^{-im s}\mathsf v(s,\tau)\dd s\,.\]
Let us consider the quadratic form
\[Q_{m}(w)=\int_{0}^{\tilde\delta}\left|(\partial_{\tau}+m)w\right|^2e^{2\tilde E \tau/h}\dd \tau\]
with boundary conditions
$w(0)=0$ and $w(\tilde \delta)=1$.

Let us notice that
\[
	Q_{m}(w) = \tilde Q_m(\rho) = \int_{0}^{\tilde\delta}\left|\partial_{\tau}\rho\right|^2e^{2\tau\tilde E/h +2(\tilde\delta-\tau)m}\dd \tau\,,
\]
where $w(\tau) = e^{m(\tilde\delta-\tau)}\rho(\tau)$ for all $\tau\in(0,\tilde \delta)$, $\rho(0) = 0$ and $\rho(\tilde\delta) = 1$. 

Since $m\mapsto \tilde Q_m(\rho)$ is an increasing function, we get that 
$
	\tilde Q_m(\rho)\geq \tilde Q_0(\rho)
$
for all $m\geq 0$ and, by Lemma \ref{lem.opt}, 
\[Q_{m}(w)\geq \Lambda_{0}(h)\,,\]
where
\[\Lambda_{0}(h)=\frac{2\tilde E/h}{1-e^{-2\tilde \delta \tilde E/h}}\geq 0\,.\]
We get, by forgetting the negative $m$, we find
\[
4h^2\int_\Omega e^{-2\phi/h}|\partial_{\overline z}v|^2\dd x
\geq 2\pi h^2\Lambda_{0}(h)\sum_{m\geq 0}|{\mathsf{v}}_{m} (\delta)|^2
=h^2\Lambda_{0}(h)\norm{\Pi_+ \left( v \circ F\right)}_{L^2(\pa D(0,1-\delta))}^2\,.
\]
\end{enumerate}
\end{proof}


In the following, we choose $\delta = h^{3/4}$.

Using Proposition \ref{prop.approxH1norm}, we show in the following lemma that we can replace $v_{h}$ by $\widetilde\Pi_{h,\delta}v_h$ in Lemma \ref{lem:lowbound_annular_energy}.
\begin{lemma}\label{lem:lowbound_4}
	Assume that $\delta = h^{3/4}$ and that $\alpha\in\left(\frac{1}{3},\frac{1}{2}\right)$. Then,
	\[
		 2 \theta_0 he^{2\phi_{\min}/h}\norm{\widetilde\Pi_{h,\delta}v_h}_{\mathscr{H}^2(\Omega_\delta)}^2(1+o(1))
		 \leq
		\lambda_k(h)\norm{e^{-\frac{1}{2h}\mathsf{Hess}_{x_{\min}}\phi (x-x_{\min}, x-x_{\min})}\widetilde\Pi_{h,\delta}v_h}_{L^2(D(x_{\min},h^{\alpha}))}^2\,,
	\]
	where 
	\[\theta_0 = \frac{\min_{\pa D(0,1)}|F'(y)|\pa_{\n}\phi(F(y))}{\max_{\pa D(0,1)}|F'(y)|\pa_{\n}\phi(F(y))}\in(0,1]\,,\]
	and where we used the notation 
	\[\|w\|^2_{\mathscr{H}^2(\Omega_{\delta})}:=\int_{\partial D(0,1-\delta)} |w\circ F|^2(\partial_{\n}\phi\circ F) |F'|\dd s\,.\]
\end{lemma}
\begin{remark}
Taking $\delta=0$ in the definition of $\|w\|^2_{\mathscr{H}^2(\Omega_{\delta})}$ above, gives
\[
	\|w\|^2_{\mathscr{H}^2(\Omega_{0})} = \int_{\partial D(0,1)} |w\circ F|^2(\partial_{\n}\phi\circ F) |F'|\dd s = \int_{\partial \Omega}|w|^2\partial_\n\phi \dd y = \Nh(w)^2\,,
\]
for $w\in \mathscr{H}^2(\Omega)$.
\end{remark}

\begin{proof}[Proof of Lemma \ref{lem:lowbound_4}]
	\begin{enumerate}[\rm i.]
	\item
	From Lemma \ref{lem:lowbound_annular_energy} and the definition of $v_{h}$, we have
	\[\begin{split}
		 2E h\norm{\Pi_+ \left(v_h\circ F\right)}_{L^2(\pa D(0,1-\delta))}^2(1+o(1))
		 &\leq h^2\int_{\Omega}e^{-2\phi/h}|\partial_{\overline{z}}v_h|^2\dd x\\
		 &\leq \lambda_{k}(h)\int_{\Omega}e^{-2\phi/h}|v_h|^2\dd x\,.
	\end{split}
	\]
	Thus, by Lemma \ref{lem.normL2},
	\begin{equation}\label{eq:holobord1}
		2E h\norm{\Pi_+ \left(v_h\circ F\right)}_{L^2(\pa D(0,1-\delta))}^2(1+o(1))\leq  \lambda_{k}(h)\int_{D(x_{\rm min},h^{\alpha})}e^{-2\phi/h}|v_h|^2\dd x\,.
	\end{equation}
	\item
	 Consider $\check\Pi_{h,\delta}$ is the orthogonal projection on $\mathscr{H}(D(0,1-\delta))$ for the scalar product $L^2(e^{-2\check\phi/h}\dd y)$.  Note that $\check\Pi_{h,\delta}\Pi_{+}= \Pi_{+}\check\Pi_{h,\delta} = \check\Pi_{h,\delta}$ (see Notation \ref{not.szego}).
	Let us now replace $\Pi_{+}$ by $\check\Pi_{h,\delta}$.	
	Proposition \ref{prop.approxH1norm} ensures that
	\[\begin{split}
		\|e^{-\check\phi/h}(\mathrm{Id}-\check\Pi_{h,\delta})v_h\circ F\|_{L^2(\partial D(0,1-\delta))}
		&\leq C h^{-2}\sqrt{\lambda_{k}(h)}\|e^{-\phi/h}v_h\|_{L^2(\Omega_{\delta})}\,,
	\end{split}\]
	Using the Taylor expansion of $\check\phi$ near the boundary and \eqref{eq.normal}, we have, on $\partial D(0,1-\delta)$, 
	\[
		e^{-\check\phi/h}\geq (1+o(1))e^{Eh^{-1/4}}\,,
	\]
	so that 
	\begin{equation}\label{eq.approx-bord}
		\|(\mathrm{Id}-\check\Pi_{h,\delta})v_h\circ F\|_{L^2(\partial D(0,1-\delta))}\leq C h^{-2} \sqrt{\lambda_{k}(h)} e^{-Eh^{-1/4}}\|e^{-\phi/h}v_h\|_{L^2(\Omega_{\delta})} \,.
	\end{equation}
	Since $\Pi_{+}$ is a projection and $\check \Pi_{h,\delta}$ is valued in the holomorphic functions,
	\[\begin{split}
		\|(\mathrm{Id}-\check\Pi_{h,\delta})v_h\circ F\|_{L^2(\partial D(0,1-\delta))}
		&\geq \|\Pi_+(\mathrm{Id}-\check\Pi_{h,\delta})v_h\circ F\|_{L^2(\partial D(0,1-\delta))} 
		\\&
		 \geq\|\Pi_+v_h\circ F-\check\Pi_{h,\delta}v_h\circ F\|_{L^2(\partial D(0,1-\delta))} 
		\\&\geq|\|\Pi_+v_h\circ F\|_{L^2(\partial D(0,1-\delta))}-\|\check\Pi_{h,\delta}v_h\circ F\|_{L^2(\partial D(0,1-\delta))}|\,.
	\end{split}\]
Then, with \eqref{eq.approx-bord},
	\[
		\|\Pi_+v_h\circ F\|_{L^2(\partial D(0,1-\delta))}\geq \|\check\Pi_{h,\delta}v_h\circ F\|_{L^2(\partial D(0,1-\delta))}- \mathscr{O}(h^\infty)\sqrt{\lambda_k(h)}\|e^{-\phi/h}v_h\|_{L^2(\Omega_{\delta})}\,.
	\]
By \eqref{eq:holobord1} and Lemma \ref{lem.normL2},
	\begin{equation}\label{eq.amelior}
		\sqrt{2Eh}\|\check\Pi_{h,\delta}v_h\circ F\|_{L^2(\partial D(0,1-\delta))}(1+o(1))\leq 
		\sqrt{\lambda_k(h)}\|e^{-\phi/h}v_h\|_{L^2(D(x_{\min},h^{\alpha}))}\,.
\end{equation}
Thus, coming back to $\Omega_{\delta}$ (without forgetting the Jacobian of $F$),
\[			\sqrt{2Eh}\||F'(F(\cdot))|^{-1/2}\widetilde \Pi_{h,\delta}v_h\|_{L^2(\partial\Omega_\delta)}(1+o(1))\leq 
		\sqrt{\lambda_k(h)}\|e^{-\phi/h}v_h\|_{L^2(D(x_{\min},h^{\alpha}))}\,.\]
Then, by using the (weighted) Hardy norm, we have
\begin{equation}\label{eq.intermediate}
\sqrt{2\theta_{0}h}\|\widetilde \Pi_{h,\delta}v_h\|_{\mathscr{H}^2(\Omega_\delta)}(1+o(1))\leq 
		\sqrt{\lambda_k(h)}\|e^{-\phi/h}v_h\|_{L^2(D(x_{\min},h^{\alpha}))}\,.
		\end{equation}
	\item
	Using Proposition \ref{prop.approxH1norm} and Lemma \ref{lem.normL2}, we get
	\[\begin{split}
		&\|e^{-\phi/h}v_h\|_{L^2(D(x_{\rm min},h^{\alpha}))}
		\\&
		\leq\|e^{-\phi/h}\widetilde\Pi_{h,\delta}v_h\|_{L^2(D(x_{\rm min},h^{\alpha}))} +\|e^{-\phi/h}(\mathrm{Id}-\widetilde\Pi_{h,\delta})v_h\|_{L^2(D(x_{\min},h^{\alpha}))}
		\\&
		\leq\|e^{-\phi/h}\widetilde\Pi_{h,\delta}v_h\|_{L^2(D(x_{\rm min},h^{\alpha}))} +\|e^{-\phi/h}(\mathrm{Id}-\widetilde\Pi_{h,\delta})v_h\|_{L^2(\Omega_\delta)}
		\\&
		\leq\|e^{-\phi/h}\widetilde\Pi_{h,\delta}v_h\|_{L^2(D(x_{\rm min},h^{\alpha}))} +C h^{-\frac{1}{2}}\sqrt{\lambda_{k}(h)} \|e^{-\phi/h}v_h\|_{L^2(D(x_{\min},h^{\alpha}))}\,.
	\end{split}\]
Combing this with \eqref{eq.intermediate} and Proposition \ref{prop:upper-bound}, we find
\[
		 2 \theta_0 h\norm{\widetilde\Pi_{h,\delta}v_h}_{\mathcal{H}^2(\Omega_\delta)}^2(1+o(1))
		 \leq
		\lambda_k(h)\|e^{-\phi/h}\widetilde\Pi_{h,\delta}v_h\|_{L^2(D(x_{\rm min},h^{\alpha}))}^2\,.
	\]
\item Using the Taylor expansion of $\phi$ at $x_{\rm min}$, we get, for all for $x\in D(x_{\rm min},h^{\alpha})$ ,
\[
	\frac{\phi(x) -\phi_{\rm min}}{h} = \frac{1}{2h}\mathsf{Hess}_{x_{\min}}\phi (x-x_{\rm min}, x-x_{\rm min}) + \mathscr{O}(h^{3\alpha-1})\,,
\]
and the conclusion follows.

\end{enumerate}
\end{proof}

\begin{remark}
Lemma \ref{lem:lowbound_4} shows in particular that
\[2(1+o(1))\theta_{0}h\tilde\lambda_{k}(h)\leq \lambda_{k}(h)\,,\]
where
\[\tilde\lambda_{k}(h)=\inf_{\begin{array}{c}V\subset \mathscr{H}^2(\Omega_\delta)\\\dim V=k\end{array}}\sup_{v\in V\setminus\{0\}}
		\frac{\norm{v}_{\mathcal{H}^2(\Omega_\delta)}^2}{\norm{e^{- \frac{1}{2h}\mathsf{Hess}_{x_{\min}}\phi (x-x_{\rm min}, x-x_{\rm min}) }v}_{L^2(D(x_{\rm min},h^{\alpha}))}^2}\,.\]
In the next section, we will essentially provide a lower bound of $\tilde\lambda_{k}(h)$.
	Note that if we could replace $\mathscr{H}^2(\Omega_\delta)$ by the set of polynomials, then, we would get the bound presented in Remark \ref{rem.naive}. Nevertheless, there is no hope to do so since in general,
	\[
		\disth((z-z_{\min})^{k-1},\mathscr{H}^2_{k}(\Omega))<\Nb((z-z_{\min})^{k-1})\,,
	\]
	(This inequality is an equality in the radial case).
	We still have to work to get the lower bound of Theorem \ref{theo.main}.
\end{remark}

\subsection{Reduction to a polynomial subspace: Proof of Proposition \ref{prop.lowbd}}\label{sec.final}
We can now prove Proposition \ref{prop.lowbd}.

\begin{enumerate}[\rm i.]

\item By using \eqref{eq.CCS}, there exists $C>0, h_{0}>0$ such that, for all $h\in(0,h_{0})$, for all $w\in\mathscr{H}^2(\Omega_{\delta})$,  all $z_0\in D(x_{\min},h^{\alpha})$, and all $n\in \{0,\dots k\}$,
\begin{equation}\label{eq.CCS2}
|w^{(n)}(z_{0})|\leq C\|w\|_{\mathscr{H}^2(\Omega_{\delta})}\,.
\end{equation}
Let us define, for all $w\in \mathscr{H}^2(\Omega_{\delta})$,
\[N_{h}(w)=\norm{e^{-\frac{1}{2h}\mathsf{Hess}_{x_{\min}}\phi (x-x_{\min}, x-x_{\min})}w}_{L^2(D(x_{\min},h^{\alpha}))}\,.\]
We let $w_{h}=\widetilde\Pi_{h,\delta}v_h$. By the Taylor formula, we can write
\[w_{h}=\mathrm{Tayl}_{k-1}w_{h}+R_{k-1}(w_{h})\,,\]
where
\[\mathrm{Tayl}_{k-1}w_{h}=\sum_{n=0}^{k-1}\frac{w^{(n)}_{h}(z_{\min})}{n!}(z-z_{\min})^n\,,\]
and, for all $z\in D(z_{\min}, h^\alpha)$,
\[|R_{k-1}(w_{h})(z)|\leq C|z-z_{\min}|^k\sup_{D(z_{\min}, h^\alpha)}|w_{h}^{(k)}|\,.\]
With \eqref{eq.CCS2} and a rescaling, the Taylor remainder satisfies
\[N_{h}(R_{k-1}(w_{h}))\leq Ch^{\frac{k}{2}}h^{\frac{1}{2}}\|w_{h}\|_{\mathscr{H}^2(\Omega_{\delta})}\,.\]
Thus, by the triangle inequality,
\[N_{h}(w_{h})\leq N_{h}(\mathrm{Tayl}_{k-1}w_{h})+Ch^{\frac{k}{2}}h^{\frac{1}{2}}\|w_{h}\|_{\mathscr{H}^2(\Omega_{\delta})}\,.\]
Thus, with Lemma \ref{lem:lowbound_4}, we get
\[(1+o(1))e^{\phi_{\min}/h}\sqrt{2\theta_{0}h}\|w_{h}\|_{\mathscr{H}^2(\Omega_{\delta})}\leq \sqrt{\lambda_{k}(h)}N_{h}(\mathrm{Tayl}_{k-1}w_{h})+C\sqrt{\lambda_{k}(h)}h^\frac{1+k}{2}\|w_{h}\|_{\mathscr{H}^2(\Omega_{\delta})}\,,\]
so that, thanks to Proposition \ref{prop:upper-bound},
\begin{equation}\label{eq.ineqTaylor}
(1+o(1))e^{\phi_{\min}/h}\sqrt{2\theta_{0}h}\|w_{h}\|_{\mathscr{H}^2(\Omega_{\delta})}\leq \sqrt{\lambda_{k}(h)}N_{h}(\mathrm{Tayl}_{k-1}w_{h})\leq \sqrt{\lambda_{k}(h)}\hat N_{h}(\mathrm{Tayl}_{k-1}w_{h})\,,
\end{equation}
with 
\[\hat N_{h}(w)=\norm{e^{-\frac{1}{2h}\mathsf{Hess}_{x_{\min}}\phi (x-x_{\min}, x-x_{\min})}w}_{L^2(\RR^2)}\,.\]
This inequality shows in particular that $\mathrm{Tayl}_{k-1}\widetilde\Pi_{\delta,h}$ is injective on $\mathscr{E}_{h}$ and
\begin{equation}\label{eq.dim-k}
\mathrm{dim}\mathrm{Tayl}_{k-1}(\widetilde\Pi_{\delta,h}\mathscr{E}_{h})=k\,.
\end{equation}

\item
Let us recall that
\[\mathscr{H}^2_{k}(\Omega_{\delta})=\{\psi\in\mathscr{H}^2(\Omega_{\delta}) : \forall n\in\{0,\ldots, k-1\}\,, \psi^{(n)}(x_{\min})=0\}\,.\]
Since $(w_{h}-\mathrm{Tayl}_{k-1}w_{h})\in\mathscr{H}^2_{k}(\Omega_{\delta})$, we have, by the triangle inequality,
\[
	\begin{split}
		&\|w_{h}\|_{\mathscr{H}^2(\Omega_{\delta})}
		\geq
		\left\| \frac{w_{h}^{(k-1)}(z_{\min})}{(k-1)!}(z-z_{\min})^{k-1}+(w_{h}-\mathrm{Tayl}_{k-1}w_{h})\right\|_{\mathscr{H}^2(\Omega_{\delta})}
		-\left\|\mathrm{Tayl}_{k-2}w_{h} \right\|_{\mathscr{H}^2(\Omega_{\delta})}
		\\
		&\geq
		\frac{|w_{h}^{(k-1)}(z_{\min})|}{(k-1)!}\disthd((z-z_{\min})^{k-1},\mathscr{H}^2_{k}(\Omega_{\delta}))
		-\left\|\mathrm{Tayl}_{k-2}w_{h} \right\|_{\mathscr{H}^2(\Omega_{\delta})}\,,
	\end{split}
\]
where
\[\begin{split}
	&\disthd((z-z_{\min})^{k-1},\mathscr{H}^2_{k}(\Omega_{\delta})) 
	\\&\qquad= \inf\left\{
		\left\| (z-z_{\min})^{k-1}-Q(z)\right\|_{\mathscr{H}^2(\Omega_{\delta})}\,, \mbox{ for all }Q\in \mathscr{H}^2_{k}(\Omega_{\delta})
	\right\}\,.
\end{split}\]

Using again the triangle inequality,
\[\|\mathrm{Tayl}_{k-2}w_{h}\|_{\mathscr{H}^2(\Omega_{\delta})}\leq C\sum_{n=0}^{k-2}|w_{h}^{(n)}(z_{\min})|\,.\]
Moreover,
\[\begin{split}
\sum_{n=0}^{k-2}|w_{h}^{(n)}(z_{\min})|\leq h^{-\frac{k-2}{2}}\sum_{n=0}^{k-2}h^{\frac{n}{2}}|w_{h}^{(n)}(z_{\min})|&\leq h^{-\frac{k-2}{2}}\sum_{n=0}^{k-1}h^{\frac{n}{2}}|w_{h}^{(n)}(z_{\min})|\\
&\leq Ch^{-\frac{k-2}{2}} h^{-\frac{1}{2}}\hat N_{h}(\mathrm{Tayl}_{k-1}w_{h})\,,
\end{split}
\]
where we used the rescaling property
\begin{equation}\label{eq.rescaling.Nh}
\hat N_{h}\left(\sum_{n=0}^{k-1} c_{n}(z-z_{\min})^n\right)=h^{\frac{1}{2}}\hat N_{1}\left(\sum_{n=0}^{k-1} c_{n}h^{\frac{n}{2}}(z-z_{\min})^n\right)\,,
\end{equation}
and the equivalence of the norms in finite dimension:
\[\exists C>0\,,\forall d\in\mathbb{C}^k\,,\quad C^{-1}\sum_{n=0}^{k-1}|d_{n}|\leq \hat N_{1}\left(\sum_{n=0}^{k-1} d_{n}(z-z_{\min})^n\right) \leq C\sum_{n=0}^{k-1}|d_{n}|\,.\]
We find
\begin{multline*}
\|w_{h}\|_{\mathscr{H}^2(\Omega_{\delta})}\geq \frac{|w_{h}^{(k-1)}(z_{\min})|}{(k-1)!}\disthd((z-z_{\min})^{k-1},\mathscr{H}^2_{k}(\Omega_{\delta}))\\
-Ch^{-\frac{k-2}{2}} h^{-\frac{1}{2}}\hat N_{h}(\mathrm{Tayl}_{k-1}w_{h})\,,
\end{multline*}
and thus, by \eqref{eq.ineqTaylor},
\begin{multline}\label{eq.ineqTaylor2}
(1+o(1))e^{\phi_{\min}/h}\sqrt{2\theta_{0}h}\frac{|w_{h}^{(k-1)}(z_{\min})|}{(k-1)!}\disthd((z-z_{\min})^{k-1},\mathscr{H}^2_{k}(\Omega_{\delta}))\\
\leq \left(\sqrt{\lambda_{k}(h)}+Ch^{\frac{2-k}{2}}e^{\phi_{\min}/h}\right)\hat N_{h}(\mathrm{Tayl}_{k-1}w_{h})\,.
\end{multline}

\item Since we have \eqref{eq.dim-k}, we deduce that
\begin{multline}\label{eq.ineqTaylor4}
(1+o(1))e^{\phi_{\min}/h}\sqrt{2\theta_{0}h}\disthd((z-z_{\min})^{k-1},\mathscr{H}^2_{k}(\Omega_{\delta}))
\sup_{c\in\mathbb{C}^k} \frac{|c_{k-1}|}{\hat N_{h}(\sum_{n=0}^{k-1}c_{n}(z-z_{\min})^n)}\\
\leq \sqrt{\lambda_{k}(h)}+Ch^{\frac{2-k}{2}}e^{\phi_{\min}/h}\,.
\end{multline}
By \eqref{eq.rescaling.Nh}, we infer
\[h^{\frac{1}{2}}\sup_{c\in\mathbb{C}^k} \frac{|c_{k-1}|}{\hat N_{h}(\sum_{n=0}^{k-1}c_{n}(z-z_{\min})^n)}=\sup_{c\in\mathbb{C}^k} \frac{h^{\frac{1-k}{2}}|c_{k-1}|}{\hat N_{1}(\sum_{n=0}^{k-1}c_{n}(z-z_{\min})^n)}\,.\]
Since $\hat N_{1}$ is related to the Segal-Bargmann norm $N_{\mathcal{B}}$ via a translation, and recalling Notation \ref{not.Pn}, we get 
\[\sup_{c\in\mathbb{C}^k} \frac{|c_{k-1}|}{\hat N_{1}(\sum_{n=0}^{k-1}c_{n}(z-z_{\min})^n)}=\sup_{c\in\mathbb{C}^k} \frac{|c_{k-1}|}{N_{\mathcal{B}}(\sum_{n=0}^{k-1}c_{n}z^n)}=\frac{1}{N_{\mathcal{B}}(P_{k-1})}\,.\]
Thus,
\begin{equation}\label{eq.ineqTaylor5}
(1+o(1))h^{\frac{1-k}{2}}e^{\phi_{\min}/h}\sqrt{2\theta_{0}}\frac{\disthd((z-z_{\min})^{k-1},\mathscr{H}^2_{k}(\Omega_{\delta}))}{N_{\mathcal{B}}(P_{k-1})}\leq \sqrt{\lambda_{k}(h)}\,.
\end{equation}

\item
Since $\Omega$ is regular enough, the Riemann mapping theorem ensures that
\[
	\lim_{h\to 0}\disthd\left((z-z_{\rm min})^{k-1},\mathscr{H}^2_{k}(\Omega_{\delta})\right) 
	= \disth\left((z-z_{\rm min})^{k-1},\mathscr{H}^2_k( \Omega)\right)\,.
\]
The conclusion follows.
\end{enumerate}

\subsection{Proof of Corollary \ref{cor.main}}\label{sec.Riemann}
We recall Notation \ref{not.RM} where $F$, $c_1$ and $c_2$ are defined. Let us notice that we can choose $F$ such that $F(0) = x_{\rm min}$.

For all $v\in H^1_{0}(\Omega)$, we let $\check v=v\circ F\in H^1_{0}(D(0,1))$, and we get, 
\[\frac{1}{c_{2}}\frac{\int_{D(0,1)}e^{-2\check\phi/h}|\partial_{\overline{y}}\check v|^2\dd y}{\int_{D(0,1)}e^{-2\check\phi/h}|\check v|^2\dd y}\leq\frac{\int_{D(0,1)}e^{-2\check\phi/h}|\partial_{\overline{y}}\check v|^2\dd y}{\int_{D(0,1)}e^{-2\check\phi/h}|\check v|^2|F'(y)|^2\dd y}=\frac{\int_{\Omega} e^{-2\phi/h}|\partial_{\overline{z}}v|^2\dd x}{\int_{\Omega}e^{-2\phi/h}|v|^2\dd x}\,,\]
where $\check\phi=\phi\circ F$ has a unique and non-degenerate minimum at $y=0$ and $\check\phi(0)=\phi_{\min}$. 

In the same way, we get
\[\frac{\int_{\Omega} e^{-2\phi/h}|\partial_{\overline{z}}v|^2\dd x}{\int_{\Omega}e^{-2\phi/h}|v|^2\dd x}\leq \frac{1}{c_{1}}\frac{\int_{D(0,1)}e^{-2\check\phi/h}|\partial_{\overline{y}}\check v|^2\dd y}{\int_{D(0,1)}e^{-2\check\phi/h}|\check v|^2\dd y}\,.\]

These inequalities, the min-max principle, and Theorem \ref{theo.main} imply Corollary \ref{cor.main}.

\appendix
\section{A unidimensional optimization problem}
The goal of this section is to minimize, for each fixed $s$, the quantity,
\[	
	\int_{0}^{\eps} e^{2t\partial_{\textbf{n}}\phi(s)/h}\left|\partial_{t}\rho\right|^2 \dd t\,.
\]
This leads to the following lemma.
\begin{lemma}\label{lem.opt}
For $\alpha$ and $\eps>0$, we let $I=(0,\eps)$ and 
\[\mathscr{V}=\{\rho\in H^1(I) : \rho(0)=0\,,\quad \rho(\eps)=1\}\,.\]
Let us consider, for all $\rho\in\mathscr{V}$,
\[F_{\alpha,\eps}(\rho):=\int_{0}^\eps e^{\alpha \ell}|\rho'(\ell)|^2\dd\ell\,.\]
\begin{enumerate}[{\rm (a)}]
	\item The minimization problem 
	\[
		\inf\{F_{\alpha,\eps}(\rho)\,, \rho\in\mathscr{V} \}
	\]
	has a unique minimizer 
	\[\rho_{\alpha,\eps}(\ell)=\frac{1-e^{-\alpha\ell}}{1-e^{-\alpha\eps}}\,,\]
	\item We have
	\[
		\inf\{F_{\alpha,\eps}(\rho)\,, \rho\in\mathscr{V} \} = \frac{\alpha}{1-e^{-\eps\alpha}}\,,
	\]
	\item Let $c_0>0$. Assume $1-e^{-\alpha\eps}\geq c_0$. There exists $C>0$ such that
\[
	\int_0^\eps e^{\alpha \ell}|\pa_\alpha \rho_{\alpha,\eps}|^2\dd \ell \leq C\left(
		\alpha^{-3} + e^{-\alpha\eps}\eps^2\alpha^{-1}
	\right)\,.
\]

\end{enumerate}
\end{lemma}

\begin{proof}
\begin{enumerate}[\rm i.]
	\item 
	Since $\alpha>0$, we have that $F_{\alpha,\eps}(\rho)\geq \int_0^\eps|\rho'(\ell)|^2\dd \ell$ for all $\rho\in \mathscr{V}$.
	There exists $C>0$ such that, for all $\rho\in\mathscr{V}$,
	 \[ \int_0^\eps|\rho'(\ell)|^2\dd \ell\geq C\int_0^{\eps}|\rho(\ell)|^2\dd \ell\,.\]

	
	This ensures that any minimizing sequence $(\rho_n)_{n\in \NN}\subset  \mathscr{V}$  is bounded in $H^1(I)$ and any $H^1$-weak limit is a minimizer of $\inf\{F_{\alpha,\eps}(\rho)\,, \rho\in\mathscr{V} \}$.
	\item
	$F_{\alpha,\eps}^{1/2}$ is an euclidian norm on $\mathscr{V}$ so that $F_{\alpha,\eps}$ is strictly convex and the minimizer is unique.
	\item
	At a minimum $\rho$, the Euler-Lagrange equation is
	\[(e^{\alpha\ell} \rho')'=0\,.\]
	Thus, there exist $(c, d)\in\mathbb{R}^2$ such that, for all $\ell\in I$,
	\[\rho(\ell)=d-c\alpha^{-1}e^{-\alpha\ell}\,,\]
	so that, with the boundary conditions we find the function $\rho_{\alpha,\eps}$. 
	\item
	We have
	\[\int_{0}^\eps e^{\alpha\ell}|\rho'(\ell)|^2\dd\ell=\alpha^2(1-e^{-\eps\alpha})^{-2}\int_{0}^\eps e^{-\alpha\ell}\dd\ell=\frac{\alpha}{1-e^{-\eps\alpha}}\,.\]
	\item
	We also have
	\[
	\pa_\alpha\rho_{\alpha,\eps}(\ell) = \frac{1}{(1-e^{-\alpha\eps})^2}\left(\ell e^{-\alpha \ell}(1-e^{-\alpha\eps})-(1-e^{-\alpha \ell})\eps e^{-\alpha \eps}\right)\,,
	\]
	 for $\ell\in(0,\eps)$ and
	 \[\begin{split}
	\int_0^\eps e^{\alpha \ell}|\pa_\alpha \rho_{\alpha,\eps}|^2\dd \ell
	&
	\leq 
	\frac{1}{(1-e^{-\alpha\eps})^4}\Big(
		\int_{0}^\eps \ell^2e^{-\alpha\ell}\dd \ell
		\left(1-e^{-\alpha \eps}\right)^2+
		\int_{0}^\eps e^{-\alpha\ell}\dd \ell
		\left(\eps e^{-\alpha\eps}\right)^2
		\\&\qquad\qquad\qquad\qquad\qquad\qquad\qquad\qquad\qquad
		+\int_{0}^\eps e^{\alpha\ell}\dd \ell
		\left(\eps e^{-\alpha\eps}\right)^2
	\Big)
	\\
	&
	\leq  C\left(
		\alpha^{-3} + e^{-\alpha\eps}\eps^2\alpha^{-1}
	\right)\,.
\end{split}\]
\end{enumerate}
\end{proof}

\section{Hopf's Lemma with Dini regularity}
In the following lemma, we present a simple proof of an extension of Hopf's Lemma to the case when $\Omega$ is Dini regular. The standard version of Hopf's Lemma given for instance in \cite[Hopf's Lemma, section 6.4.2] {evans1998partial} requires essentially $\mathscr{C}^2$ regularity.  However, the regularity can be lowered up to Dini (see \cite{MR3513140} and the references therein).
\begin{lemma}\label{lem.hopfDini}
	Let $\Omega$ being a simply connected, Dini regular, bounded open set. Then, if $\phi$ is the solution of \eqref{eq:0modeFlux}, the function $s\in\partial \Omega \mapsto \pa_\n \phi(s)$ is continuous and
	\[
		\pa_\n \phi>0\text{ on }\pa \Omega\,.
	\]
\end{lemma}
\begin{proof}
	Let $\phi$ be the solution of \eqref{eq:0modeFlux}. By the Riemann mapping Theorem \cite{P92}, there exists a bi-holomorphic map $F\colon D(0,1)\to \Omega$ such that $F'$ is continuous on $\overline{D(0,1)}$. The function $\check \phi = \phi\circ F$ is the solution of \eqref{eq:0modeFlux} on $D(0,1)$ for $\check B = |F'|^2B\circ F$. By \cite[Corollary .36]{MR737190}, we get that $\check\phi$ is $\mathscr{C}^{1,1}$ on $D(0,1)$ and Hopf's Lemma \cite[Hopf's Lemma, Section 6.4.2] {evans1998partial} ensures that $\pa_\n\check \phi>0$. The result follows from the fact that
	\[
		\pa_\n \phi = |(F^{-1})'|\pa_\n\check \phi\circ F^{-1}\,.
	\]
\end{proof}

\section{A density result}
\begin{lemma}\label{lem.density}
	Assume that $\Omega$ is bounded, simply connected and that $\partial\Omega$ is Dini-continuous. Then, the set
	$
		\mathscr{H}^2(\Omega)\cap W^{1,\infty}(\Omega)	
	$
	is dense in $\mathscr{H}^2(\Omega)$.
\end{lemma}
\begin{proof}
We recall Notation \ref{not.RM}. Let $u\in \mathscr{H}^2(\Omega)$. Then, $u\circ F = \sum_{k\geq 0}a_kz^k$ is holomorphic on $D(0,1)$ and $(a_k)_{k\geq0}\in \ell^2(\NN)$. Let $\eps\in(0,1)$. The function 
	\[
		\widetilde u_\eps : D(0,1)\ni z\mapsto u\circ F((1-\eps)z)\in \CC
	\]
	is holomorphic on $D(0,1/(1-\eps))$. We denote by $u_\eps = \widetilde u_\eps\circ F^{-1}$. We have
	\[\begin{split}
		\norm{u-u_\eps}_{\mathscr{H}^2(\Omega)}^2 
		&:= \int_{\partial \Omega}|u(x)-u_\eps(x)|^2\partial_\n \phi\dd x 
		\\&
		= \int_{\partial D(0,1)}|u\circ F(y)-u\circ F((1-\eps)y)|^2|F'(y)|\partial_\n \phi\circ F(y)\dd y
		\\&
		\leq
		c_2\|\partial_\n \phi\|_{L^\infty}\int_{\partial D(0,1)}|u\circ F(y)-u\circ F((1-\eps)y)|^2\dd y
		\\&
		\leq
		c_2\|\partial_\n \phi\|_{L^\infty}\sum_{k\geq 1}|a_k|^2|1-(1-\eps)^k|^2\,.
	\end{split}\]
	 Note that $\pa_\n \phi$ is bounded by Lemma \ref{lem.hopfDini}.
	By Lebesgue's theorem, we get that $(u_\eps)_{\eps\in(0,1)}$ converges to $u$ in $\mathscr{H}^2(\Omega)$. Let us also remark that $(u_\eps)_{\eps\in(0,1)}\subset W^{1,\infty}(\Omega)$ so that the result follows.
\end{proof}

\subsection*{Acknowledgment}
N.R. is grateful to T. Hmidi for sharing the reference \cite{P92}. L.~L.T. is grateful to S. Monniaux, J. Olivier, S. Rigat and E.H. Youssfi for useful discussions. The authors are grateful to Bernard Helffer for many useful comments. 
L.~L.T. is supported by ANR DYRAQ ANR-17-CE40-0016-01. E.S has been partially funded by Fondecyt (Chile)
project  \# 118--0355.
\bibliographystyle{abbrv}
\bibliography{biblioPauli}

\begin{thebibliography}{10}

\bibitem{MR3513140}
D.~E. Apushkinskaya and A.~I. Nazarov.
\newblock A counterexample to the {H}opf-{O}leinik lemma (elliptic case).
\newblock {\em Anal. PDE}, 9(2):439--458, 2016.

\bibitem{BR18}
Y.~Bonthonneau and N.~Raymond.
\newblock {WKB constructions in bidimensional magnetic wells}.
\newblock {\em Preprint: arXiv:1711.04475}, 2018.

\bibitem{MR2759829}
H.~Brezis.
\newblock {\em Functional analysis, {S}obolev spaces and partial differential
  equations}.
\newblock Universitext. Springer, New York, 2011.

\bibitem{duren2000theory}
P.~L. Duren.
\newblock {\em Theory of {$H^p$} Spaces}.
\newblock Dover books on mathematics. Dover Publications, 2000.

\bibitem{EKP16}
T.~Ekholm, H.~Kova{\v r}{\'\i}k, and F.~Portmann.
\newblock Estimates for the lowest eigenvalue of magnetic {L}aplacians.
\newblock {\em J. Math. Anal. Appl.}, 439(1):330--346, 2016.

\bibitem{E96}
L.~Erd{\H o}s.
\newblock Rayleigh-type isoperimetric inequality with a homogeneous magnetic
  field.
\newblock {\em Calc. Var. Partial Differential Equations}, 4(3):283--292, 1996.

\bibitem{evans1998partial}
L.~C. Evans.
\newblock {\em Partial differential equations}.
\newblock Providence, Rhode Land: American Mathematical Society, 1998.

\bibitem{FH11}
S.~Fournais and B.~Helffer.
\newblock {\em Spectral methods in surface superconductivity}, volume~77 of
  {\em Progress in Nonlinear Differential Equations and their Applications}.
\newblock Birkh\"auser Boston, Inc., Boston, MA, 2010.

\bibitem{MR737190}
D.~Gilbarg and N.~S. Trudinger.
\newblock {\em Elliptic partial differential equations of second order}, volume
  224 of {\em Grundlehren der Mathematischen Wissenschaften [Fundamental
  Principles of Mathematical Sciences]}.
\newblock Springer-Verlag, Berlin, second edition, 1983.

\bibitem{HelfferMorame}
B.~Helffer and A.~Morame.
\newblock Magnetic bottles in connection with superconductivity.
\newblock {\em J. Funct. Anal.}, 185(2):604--680, 2001.

\bibitem{HP17}
B.~Helffer and M.~Persson~Sundqvist.
\newblock On the semi-classical analysis of the ground state energy of the
  {D}irichlet {P}auli operator.
\newblock {\em J. Math. Anal. Appl.}, 449(1):138--153, 2017.

\bibitem{HP17b}
B.~Helffer and M.~Persson~Sundqvist.
\newblock On the semi-classical analysis of the groundstate energy of the
  {D}irichlet {P}auli operator in non-simply connected domains.
\newblock {\em Preprint arXiv:1702.02404}, 2017.

\bibitem{K85}
B.~Kawohl.
\newblock {\em Rearrangements and convexity of level sets in {PDE}}, volume
  1150 of {\em Lecture Notes in Mathematics}.
\newblock Springer-Verlag, Berlin, 1985.

\bibitem{K85b}
B.~Kawohl.
\newblock When are solutions to nonlinear elliptic boundary value problems
  convex?
\newblock {\em Comm. Partial Differential Equations}, 10(10):1213--1225, 1985.

\bibitem{P92}
C.~Pommerenke.
\newblock {\em Boundary behaviour of conformal maps}, volume 299 of {\em
  Grundlehren der Mathematischen Wissenschaften [Fundamental Principles of
  Mathematical Sciences]}.
\newblock Springer-Verlag, Berlin, 1992.

\bibitem{S95}
K.~M. Schmidt.
\newblock A remark on boundary value problems for the {D}irac operator.
\newblock {\em Quart. J. Math. Oxford Ser. (2)}, 46(184):509--516, 1995.

\bibitem{T92}
B.~Thaller.
\newblock {\em The {D}irac equation}.
\newblock Texts and Monographs in Physics. Springer-Verlag, Berlin, 1992.

\end{thebibliography}

 \end{document}